\documentstyle[11pt,amssymb,amsmath,amscd,amsbsy,amsfonts,amsthm,color]{article}

\input xy
\xyoption{all}

\newtheorem{thm}{Theorem}[section]
\newtheorem{lem}[thm]{Lemma}
\newtheorem{prop}[thm]{Proposition}
\newtheorem{cor}[thm]{Corollary}

\DeclareMathOperator{\Z}{\mathbb{Z}}
\DeclareMathOperator{\HF}{\mathbb{H}}
\DeclareMathOperator{\N}{\mathit{N_>}}
\DeclareMathOperator{\n}{\mathit{N_<}}
\DeclareMathOperator{\M}{\mathit{\widehat{M_{\alpha}}}}
\DeclareMathOperator{\Nn}{\mathit{N_>^{\otimes n}}}
\DeclareMathOperator{\Nm}{\mathit{N_>^{\otimes m}}}
\DeclareMathOperator{\Nmp}{\mathit{N_>^{\otimes m+1}}}
\DeclareMathOperator{\Nmi}{\mathit{N_>^{\otimes m+i}}}
\DeclareMathOperator{\Nmin}{\mathit{N_>^{\otimes m+i_{2n}}}}
\DeclareMathOperator{\Nnm}{\mathit{N_>^{\otimes n-1}}}
\DeclareMathOperator{\Nodd}{\mathit{N_>^{\otimes m+ \sum_{l \: odd} i_l}}}
\DeclareMathOperator{\nodd}{\mathit{N_>^{\otimes \sum_{l \: odd} i_l}}}
\DeclareMathOperator{\Ni}{\mathit{N_>^{\otimes i_{2n}}}}

\newcommand{\la}{\langle}
\newcommand{\ra}{\rangle}
\newcommand{\lva}{\langle\!\langle}  
\newcommand{\rva}{\rangle\!\rangle}   

\title{\textsc{Subtle characteristic classes and Hermitian forms}}
\author{Fabio Tanania}
\date{}

\begin{document}

\maketitle

\begin{abstract}
Following \cite{SV}, we compute the motivic cohomology ring of the Nisnevich classifying space of the unitary group associated to the standard split hermitian form of a quadratic extension. This provides us with subtle characteristic classes which take value in the motivic cohomology of the \v{C}ech simplicial scheme associated to a hermitian form. Comparing these new classes with subtle Stiefel-Whitney classes arising in the orthogonal case, we obtain relations among the latter ones holding in the motivic cohomology of the \v{C}ech simplicial scheme associated to a quadratic form divisible by a 1-fold Pfister form. Moreover, we present a description of the motive of the torsor corresponding to a hermitian form in terms of its subtle characteristic classes.
\end{abstract}

\section{Introduction}

The study of homotopy theory in the algebro-geometric world, which was initiated by Morel and Voevodsky in \cite{MV}, has led to very deep results such as the affirmation of \textit{Milnor Conjecture} (\cite{V2}). As a result, much attention has been devoted in the last years to transferring topological techniques into algebraic geometry.

For example, the study of classifying spaces $BG$ and their respective characteristic classes in different cohomology theories have been extremely useful in topology to approach the classification of principal $G$-bundles. In the same way, it is possible to study $G$-torsors in algebraic geometry by focusing on classifying spaces and characteristic classes in the motivic homotopic environment. We notice that here there are two different, but highly related, classifying spaces, namely the \textit{Nisnevich} and the \textit{\'etale}. For non special algebraic groups they have in general different cohomology rings and, consequently, they produce different characteristic classes. Although $G$-torsors are classified by the \'etale classifying space, it is undoubted that the Nisnevich version provides its own advantages.

Good evidence of this is provided by \cite{SV}, where torsors, Nisnevich classifying spaces and a general homotopic framework to deal with them have been deeply studied by the authors. In particular, they focus on $BO_n$, the Nisnevich classifying space of the orthogonal group associated to the standard split quadratic form, which allows to study $O_n$-torsors over the point that are in one-to-one correspondence with quadratic forms. They prove that the motivic cohomology ring with $\Z/2$-coefficients of $BO_n$ is a polynomial algebra over the motivic cohomology of the base field generated by some elements that they call \textit{subtle Stiefel-Whitney classes}. These classes are very informative, for example they are able to see if a quadratic form is in a power of the fundamental ideal of the Witt ring or not. Moreover, they are related to the $J$-invariant of quadrics introduced in \cite{V}. 

In this work we will focus instead on the unitary group $U_n(E/k)$ associated to the standard split hermitian form of a quadratic extension $E/k$. In particular, we will compute the motivic cohomology with $\Z/2$-coefficients of its Nisnevich classifying space. As in the orthogonal case, this will provide us with subtle characteristic classes which allow to approach the classification of $U_n(E/k)$-torsors over the point that are nothing else but $n$-dimensional hermitian forms of $E/k$, which in turn are in one-to-one correspondence with $2n$-dimensional quadratic forms over $k$ divisible by the norm form of the quadratic extension considered. In \cite{SV}, the computation of the motivic cohomology of $BO_n$ is conducted inductively by using fibrations with motivically Tate fibers. In our situation, new features will appear. In particular, the fibrations in the unitary case, similar to those considered in the orthogonal one, will have reduced fibers which (depending on parity) are not motivically Tate but, anyway, invertible, which will still allow the computation. These invertible motives are, not surprisingly, closely related to the \textit{Rost motive} of our quadratic extension. As a consequence, we obtain that, unlike the orthogonal case, the classifying space of the unitary group is not cellular, but it becomes one once tensored with the \v{C}ech simplicial scheme of the Pfister form of the quadratic extension. Related to this, we observe an interesting interaction between invertible objects and idempotents in Voevodsky category. It is manifested, in particular, by the fact that the cohomology of the tensor product of $BU_n(E/k)$ with the \v{C}ech simplicial scheme above mentioned happens to be a direct limit of the cohomology of $BU_n(E/k)$ tensored with powers of an invertible motive. We also note that, although studying hermitian forms is the same as studying quadratic forms divisible by a Pfister form, the understanding of the unitary case allows to trace back information from the hermitian world to the quadratic one. In particular, from the computation of the motivic cohomology of $BU_n(E/k)$ we get relations among subtle Stiefel-Whitney classes in the cohomology of the \v{C}ech simplicial scheme of the respective quadratic form divisible by a binary Pfister form. In this sense, for this special class of quadratic forms, the cohomology of $BU_n(E/k)$ is much closer to that of the \v{C}ech simplicial scheme of the torsor than the cohomology of $BO_{2n}$. 

We will now summarise the content of the sections of this text. First of all, in section $2$ we present a few notations which will be followed throughout the paper. Section $3$ is devoted to recalling some preliminary definitions and results about the category of motives over a simplicial base studied in \cite{V1}. Moreover, following \cite{SV}, we will prove some statements regarding fibrations with motivically invertible reduced fibers. In section $4$ we will deal with Nisnevich classifying spaces and some of their main features.  A study of the \v{C}ech simplicial scheme, the Rost motive of a quadratic extension and, especially, of some closely related invertible motives is presented in section $5$. The main result of the paper, namely the computation of the motivic cohomology ring of $BU_n(E/k)$ is object of section $6$. In section $7$, we will compare the classifying space of the unitary group of the split hermitian form with that of the orthogonal group of the corresponding quadratic form. As a consequence, we will present the cohomology of the first as a quotient of the second. In particular, we will relate our subtle classes to subtle Stiefel-Whitney classes arising in the orthogonal case. Finally, in section $8$, in the same fashion of \cite{SV}, we find some applications to hermitian forms. For example, we deduce relations among subtle classes in the motivic cohomology of the respective \v{C}ech simplicial scheme, see that these subtle classes distinguish the triviality of the torsor and find an expression of the motive of the torsor associated to a hermitian form.\\

\textbf{Acknowledgements.} I wish to express my sincere gratitude to my PhD supervisor Alexander Vishik for his precious help and constant encouragement throughout the preparation of this work. Moreover, I would like to thank the referee for very useful comments which helped to correct some mistakes and to improve the exposition. 

\section{Notation}

Throughout this paper we will work over a field $k$ of characteristic different from $2$.

The main categories we will consider are the category of motivic spaces $Spc(k)=\Delta^{op}Shv_{Nis}(Sm/k)$, the simplicial homotopy category ${\mathcal H}_s(k)$  constructed by Morel and Voevodsky in \cite{MV} and the triangulated category of effective motives ${\mathcal DM}^{-}_{eff}(k)$ constructed by Voevodsky in \cite{V3}. 

All motivic cohomology will be with $\Z/2$-coefficients. Moreover, we will denote by $H$ the motivic cohomology of $Spec(k)$. From a result by Voevodsky (\cite{V2}), we know that $H=K^M(k)/2[\tau]$, where $\tau$ is the generator of $H^{0,1}=\Z/2$.

Given a quadratic extension $E=k(\sqrt{\alpha})$ and an $n$-dimensional $E$-vector space $V$, an $n$-dimensional hermitian form is a map $h:V \times V \rightarrow E$ which is $E$-linear in the first factor and such that $h(v,w)=\sigma({h(w,v)})$ (where $\sigma$ is the generator of $Gal(E/k)$). It follows immediately from the definition that the diagonal part of a hermitian form takes values in $k$ and is a quadratic form. We will denote by $\widetilde{h}$ this $2n$-dimensional quadratic form over $k$ defined by $\widetilde{h}(v)=h(v,v)$ for any $v \in V$ considered as a $2n$-dimensional $k$-vector space. Moreover, notice that the quadratic form $\widetilde{h}$ just defined is divisible by $\lva \alpha \rva$, the $1$-fold Pfister form associated to $\alpha$. Indeed, more is true, namely any quadratic form over $k$ divisible by $\lva \alpha \rva$ is associated to some hermitian form, and the correspondence is bijective. In fact, given two $n$-dimensional hermitian forms $h$ and $h'$, we have that $h \cong h'$ if and only if $\widetilde{h} \cong \widetilde{h'}$ (\cite[Corollary 9.2]{Ka}).

We will express by $q_n$ the standard split quadratic form $\perp_{i=1}^n \la (-1)^{i-1} \ra$ and by $\HF$ the hyperbolic form $\la 1,-1 \ra$. Similarly, we will denote by $h_n$ the standard split hermitian form $\perp_{i=1}^n \la (-1)^{i-1} \ra$. Notice, in particular, that $\widetilde{h}_n=\lva \alpha \rva \otimes q_n$. By $U_n(E/k)$ we will mean the unitary group of invertible $n \times n$-matrices over $E$ that preserve the standard split hermitian form $h_n$. Notice that this is a linear algebraic group over $k$.

\section{Motives over a simplicial base}

We start this section by recalling a few definitions and some results about the category of motives over a simplicial base studied by Voevodsky in \cite{V1}.

Let $Y_{\bullet}$ be a smooth simplicial scheme over $k$ and $R$ a commutative ring with unity. As in \cite{V1}, let $Sm/Y_{\bullet}$ be the category whose objects are pairs $(U,j)$, where $j$ is a non-negative integer and $U$ is a smooth scheme over $Y_j$, and whose morphisms from $(U,j)$ to $(V,i)$ are pairs $(f,\theta)$, where $\theta:[i] \rightarrow [j]$ is a simplicial map, such that the following diagram
$$
\xymatrix{
U \ar@{->}[r]^{f} \ar@{->}[d] & V \ar@{->}[d]\\
Y_j \ar@{->}[r]_{Y_{\theta}} & Y_i
}
$$
commutes.

Then, we will denote by $Spc(Y_{\bullet})=\Delta^{op}Shv_{Nis}(Sm/Y_{\bullet})$ the category of motivic spaces over $Y_{\bullet}$, obtained by considering simplicial Nisnevich sheaves on $Sm/Y_{\bullet}$.

In $\cite{V1}$ the category of motives over $Y_{\bullet}$ with $R$ coefficients was constructed. We will denote this category by ${\mathcal DM}_{eff}^-(Y_{\bullet},R)$.

This category comes endowed with a sequence of functors 
$$r_i^*:{\mathcal DM}_{eff}^-(Y_{\bullet},R) \rightarrow {\mathcal DM}_{eff}^-(Y_i,R)$$
For simplicity we will write $N_i$ for $r_i^*(N)$. We recall that, for any morphism $p:Y_{\bullet} \rightarrow Y'_{\bullet}$ of smooth simplicial schemes, there is an adjoint pair
\begin{align*}
{\mathcal DM}_{eff}^-&(Y_{\bullet},R)\\
Lp^* \uparrow & \downarrow Rp_*\\
{\mathcal DM}_{eff}^-&(Y'_{\bullet},R)
\end{align*}
If moreover $p$ is smooth, then we also have the adjoint pair
\begin{align*}
{\mathcal DM}_{eff}^-&(Y_{\bullet},R)\\
Lp_{\#} \downarrow & \uparrow p^*\\
{\mathcal DM}_{eff}^-&(Y'_{\bullet},R)
\end{align*}
Besides, we will denote by $CC(Y_{\bullet})$ the simplicial set obtained by applying to $Y_{\bullet}$ the functor $CC$ which commutes with coproducts and sends any connected scheme to the point.

For any smooth simplicial scheme $Y_{\bullet}$ over $k$ we can consider the projection to the base $Y_{\bullet} \rightarrow Spec(k)$. This morphism induces the triangulated functor 
$$c^*:{\mathcal DM}_{eff}^-(k,R) \rightarrow {\mathcal DM}_{eff}^-(Y_{\bullet},R)$$
We will denote by $M_{Y_\bullet}$ the motive $c^*(M)$ for any $M \in {\mathcal DM}_{eff}^-(k,R)$.

We report below two results about the category of motives over a simplicial base which will be essential throughout this paper in order to deal with fibrations with motivically invertible reduced fibers. First of all, in $\cite{SV}$ it is proven the following proposition.

\begin{prop}\label{SV}
\cite[Proposition 3.1.5]{SV} Suppose that $H^1(CC(Y_{\bullet}),R^{\times})=0$. Let $T$ be the unit in ${\mathcal DM}_{eff}^-(k,R)$ and $N \in {\mathcal DM}_{eff}^-(Y_{\bullet},R)$ be such a motive that its graded components $N_i \in {\mathcal DM}_{eff}^-(Y_i,R)$ are isomorphic to $T_{Y_i}$ and all the structure maps $N_{\theta}:LY_{\theta}^*(N_i) \rightarrow N_j$ are isomorphisms for any simplicial map $\theta:[i] \rightarrow [j]$. Then $N$ is isomorphic to $T_{Y_{\bullet}}$.
\end{prop}

From the previous proposition we immediately deduce the following corollary which is a generalisation for all invertible motives.

\begin{cor}\label{SVI}
Suppose that $H^1(CC(Y_{\bullet}),R^{\times})=0$. Let $M$ be an invertible motive in ${\mathcal DM}_{eff}^-(k,R)$ and $N \in {\mathcal DM}_{eff}^-(Y_{\bullet},R)$ be such a motive that its graded components $N_i \in {\mathcal DM}_{eff}^-(Y_i,R)$ are isomorphic to $M_{Y_i}$ and all the structure maps $N_{\theta}:LY_{\theta}^*(N_i) \rightarrow N_j$ are isomorphisms for any simplicial map $\theta:[i] \rightarrow [j]$. Then $N$ is isomorphic to $M_{Y_{\bullet}}$.
\end{cor}
\begin{proof}
Consider the motive $N \otimes M_{Y_{\bullet}}^{-1}$ in ${\mathcal DM}_{eff}^-(Y_{\bullet},R)$. We notice that 
$$(N \otimes M_{Y_{\bullet}}^{-1})_i \cong N_i \otimes M_{Y_i}^{-1} \cong M_{Y_i} \otimes M_{Y_i}^{-1} \cong T_{Y_i}$$ 
and, for any simplicial map $\theta:[i] \rightarrow [j]$, the morphisms 
$$LY_{\theta}^*((N \otimes M_{Y_{\bullet}}^{-1})_i) \rightarrow (N \otimes M_{Y_{\bullet}}^{-1})_j$$ 
are nothing else but the isomorphisms $(LY_{\theta}^*(N_i) \rightarrow N_j) \otimes M_{Y_j}^{-1}$. Then, it follows from Proposition \ref{SV} that $N \otimes M_{Y_{\bullet}}^{-1} \cong T_{Y_{\bullet}}$, which completes the proof. 
\end{proof}

Notice that the condition $H^1(CC(Y_{\bullet}),R^{\times})=0$ is automatically satisfied if $R=\Z/2$, which is the case we will be interested in.

Before proceeding with the next results, we recall some definitions about coherence. A smooth morphism $\pi:X_{\bullet} \rightarrow Y_{\bullet}$ of smooth simplicial schemes is called smooth coherent if for any simplicial map $\theta:[i] \rightarrow [j]$ the following diagram
$$
\xymatrix{
X_j \ar@{->}[r]^{\pi_j} \ar@{->}[d]_{X_{\theta}} & Y_j \ar@{->}[d]^{Y_{\theta}}\\
X_i \ar@{->}[r]_{\pi_i} & Y_i
}
$$
is cartesian and all the $\pi_j$ are smooth. An object $N$ in ${\mathcal DM}_{eff}^-(Y_{\bullet},R)$ is called coherent if, for any simplicial map $\theta:[i] \rightarrow [j]$, the structural map $N_{\theta}:LY_{\theta}^*(N_i) \rightarrow N_j$ is an isomorphism. Let ${\mathcal DM}_{coh}^-(Y_{\bullet},R)$ be the full subcategory of ${\mathcal DM}_{eff}^-(Y_{\bullet},R)$ consisting of coherent objects. We notice that ${\mathcal DM}_{coh}^-(Y_{\bullet},R)$ is closed under taking cones and arbitrary direct sums, since $LY_{\theta}^*$ is a triangulated functor. It immediately follows from these definitions that, if $\pi$ is smooth coherent, then $L\pi_{\#}$ maps coherent objects to coherent ones and, in particular, $M(X_{\bullet} \xrightarrow{\pi} Y_{\bullet})$ belongs to ${\mathcal DM}_{coh}^-(Y_{\bullet},R)$, where $M(X_{\bullet} \xrightarrow{\pi} Y_{\bullet})$ is nothing else but the image $L\pi_{\#}(T_{X_{\bullet}})$ of the trivial Tate motive.

We present now the main technique taken from \cite{SV} we will use in our computation. This result allows to generate long exact sequences (of the same nature of Gysin sequences for sphere bundles in topology) in motivic cohomology associated to fibrations with reduced fibers which are motivically invertible.

\begin{prop}\label{Thom1}
Let $\pi:X_{\bullet} \rightarrow Y_{\bullet}$ be a smooth coherent morphism of smooth simplicial schemes over $k$ and $A$ a smooth $k$-scheme such that:\\
1) over the $0$ simplicial component $\pi$ is isomorphic to the map $Y_0 \times A \rightarrow Y_0$;\\
2) $H^1(CC(Y_{\bullet}),R^{\times})=0$;\\
3) $\widetilde{M}(A)$ is an invertible motive in ${\mathcal DM}_{eff}^-(k,R)$, where by $\widetilde{M}(A)$ we mean $Cone(M(A) \rightarrow T)[-1]$.\\
Then, $Cone(\pi) \cong \widetilde{M}(A)_{Y_{\bullet}}[1] \in {\mathcal DM}_{eff}^-(Y_{\bullet},R)$.
\end{prop}
\begin{proof}
In the motivic category ${\mathcal DM}_{eff}^-(Y_{\bullet},R)$ we have a distinguished triangle 
$$M(X_{\bullet} \xrightarrow{\pi} Y_{\bullet}) \rightarrow T_{Y_{\bullet}} \rightarrow Cone(\pi) \rightarrow M(X_{\bullet} \xrightarrow{\pi} Y_{\bullet})[1]$$ 
By condition $1$ and from the fact that our morphism is smooth coherent it follows that it is the projection over any simplicial component. So, we obtain that the morphism $\pi_i:Y_i \times A \cong X_i \rightarrow Y_i$ induces in ${\mathcal DM}_{eff}^-(Y_i,R)$ the map $M(A)_{Y_i} \rightarrow T_{Y_i}$ for any $i$. Thus, $Cone(\pi)_i \cong \widetilde{M}(A)_{Y_i}[1] $ is an invertible motive in ${\mathcal DM}_{eff}^-(Y_i,R)$. Since $M(X_{\bullet} \xrightarrow{\pi} Y_{\bullet})$ and $T_{Y_{\bullet}}$ belong to ${\mathcal DM}_{coh}^-(Y_{\bullet},R)$, we have that $Cone(\pi)$ is in ${\mathcal DM}_{coh}^-(Y_{\bullet},R)$ and, by Corollary \ref{SVI}, we get that $Cone(\pi) \cong \widetilde{M}(A)_{Y_{\bullet}}[1]$ in ${\mathcal DM}_{eff}^-(Y_{\bullet},R)$, as we aimed to show. 
\end{proof}

Later, we will also need the following result about functoriality of the isomorphism found in the previous proposition.

\begin{prop}\label{Thom2}
Let $\pi:X_{\bullet} \rightarrow Y_{\bullet}$ and $\pi':X'_{\bullet} \rightarrow Y'_{\bullet}$ be smooth coherent morphisms of smooth simplicial schemes over $k$ and $A$ a smooth $k$-scheme that satisfies all conditions from the previous proposition with respect to $\pi'$ and such that the following diagram is cartesian with all morphisms smooth
$$
\xymatrix{
X_{\bullet} \ar@{->}[r]^{\pi} \ar@{->}[d]_{p_X} & Y_{\bullet} \ar@{->}[d]^{p_Y}\\
X'_{\bullet} \ar@{->}[r]_{\pi'} & Y'_{\bullet}
}
$$
	Then, the induced square of motives in the category ${\mathcal DM}_{eff}^-(Y'_{\bullet},R)$ extends uniquely to a morphism of triangles where $Lp_{Y\#}Cone(\pi) \rightarrow Cone(\pi')$ is given by $M(p_Y) \otimes id_{{\widetilde M}(A)_{Y'_{\bullet}}[1]}$.
\end{prop}
\begin{proof}
First of all, we notice that in ${\mathcal DM}_{eff}^-(Y'_{\bullet},R)$ there is the following morphism of distinguished triangles
$$
\xymatrix{
Lp_{Y\#}M(X_{\bullet} \xrightarrow{\pi} Y_{\bullet}) \ar@{->}[r] \ar@{->}[d]_{M(p_X)} & Lp_{Y\#}T_{Y_{\bullet}} \ar@{->}[r] \ar@{->}[d]^{M(p_Y)} & Lp_{Y\#}Cone(\pi) \cong Lp_{Y\#}{\widetilde M}(A)_{Y_{\bullet}}[1] \ar@{->}[r] \ar@{->}[d]^{M(p)} & Lp_{Y\#}M(X_{\bullet} \xrightarrow{\pi} Y_{\bullet})[1] \ar@{->}[d]^{M(p_X)[1]}\\
M(X'_{\bullet} \xrightarrow{\pi'} Y'_{\bullet}) \ar@{->}[r] & T_{Y'_{\bullet}} \ar@{->}[r] & Cone(\pi') \cong {\widetilde M}(A)_{Y'_{\bullet}}[1] \ar@{->}[r] & M(X'_{\bullet} \xrightarrow{\pi'} Y'_{\bullet})[1]
}
$$
where the isomorphisms in the diagram are due to Proposition \ref{Thom1}. Once restricted to the $0$ simplicial component the previous diagram becomes in ${\mathcal DM}_{eff}^-(Y'_0,R)$
$$
\xymatrix{
	Lp_{Y_0\#}M(A)_{Y_0} \ar@{->}[r] \ar@{->}[d]_{M(p_{Y_0}) \otimes id_{M(A)_{Y'_0}}} & Lp_{Y_0\#}T_{Y_0} \ar@{->}[r] \ar@{->}[d]^{M(p_{Y_0})} & Lp_{Y_0\#}{\widetilde M}(A)_{Y_0}[1] \ar@{->}[r] \ar@{->}[d]^{M(p_0)} & Lp_{Y_0\#}M(A)_{Y_0}[1] \ar@{->}[d]^{M(p_{Y_0}) \otimes id_{M(A)_{Y'_0}}[1]}\\
	M(A)_{Y'_0}  \ar@{->}[r] & T_{Y'_0} \ar@{->}[r] & {\widetilde M}(A)_{Y'_0}[1] \ar@{->}[r] & M(A)_{Y'_0}[1]
}
$$
Note that 
$$Hom_{{\mathcal DM}_{eff}^-(Y'_0,R)}(Lp_{Y_0\#}{\widetilde M}(A)_{Y_0}[1],T_{Y'_0}) \cong Hom_{{\mathcal DM}_{eff}^-(Y_0,R)}({\widetilde M}(A)_{Y_0}[1],p_{Y_0}^*T_{Y'_0}) \cong $$
$$Hom_{{\mathcal DM}_{eff}^-(Y_0,R)}({\widetilde M}(A)_{Y_0}[1],T_{Y_0}) \cong Hom_{{\mathcal DM}_{eff}^-(k,R)}({\widetilde M}(Y_0 \times A)[1],T) \cong  0$$
since $Y_0 \times A$ is a smooth scheme over $k$ and, so, has no cohomology in bidegree $(0)[-1]$. From this we deduce that $M(p_0)$ must be $M(p_{Y_0}) \otimes id_{{\widetilde M}(A)_{Y'_0}[1]}$. 

At this point we notice that both $M(p)$ and $M(p_Y) \otimes id_{{\widetilde M}(A)_{Y'_{\bullet}}[1]}$ belong to 
$$Hom_{{\mathcal DM}_{eff}^-(Y'_{\bullet},R)}(Lp_{Y\#}{\widetilde M}(A)_{Y_{\bullet}}[1],{\widetilde M}(A)_{Y'_{\bullet}}[1]) \cong Hom_{{\mathcal DM}_{eff}^-(Y_{\bullet},R)}({\widetilde M}(A)_{Y_{\bullet}},p_Y^*{\widetilde M}(A)_{Y'_{\bullet}}) \cong$$
$$Hom_{{\mathcal DM}_{eff}^-(Y_{\bullet},R)}({\widetilde M}(A)_{Y_{\bullet}},{\widetilde M}(A)_{Y_{\bullet}}) \cong H^{0,0}(Y_{\bullet},R)$$
since ${\widetilde M}(A)_{Y_{\bullet}}$ is an invertible motive. Similarly $M(p_0)=M(p_{Y_0})\otimes id_{{\widetilde M}(A)_{Y'_0}[1]}$ belongs to 
$$Hom_{{\mathcal DM}_{eff}^-(Y'_0,R)}(Lp_{Y_0\#}{\widetilde M}(A)_{Y_0}[1],{\widetilde M}(A)_{Y'_0}[1]) \cong Hom_{{\mathcal DM}_{eff}^-(Y_0,R)}({\widetilde M}(A)_{Y_0},p_{Y_0}^*{\widetilde M}(A)_{Y'_0}) \cong$$
$$Hom_{{\mathcal DM}_{eff}^-(Y_0,R)}({\widetilde M}(A)_{Y_0},{\widetilde M}(A)_{Y_0}) \cong H^{0,0}(Y_0,R)$$
Now, since the complex $R(0)$ is quasi isomorphic to the constant sheaf $R$ concentrated in degree $0$, we have that $H^{0,0}(Y_{\bullet},R)$ is just the sheaf cohomology group $H^0(Y_{\bullet},R)$. From \cite[sections 5.1 and 5.2]{D}, one knows how to compute the sheaf cohomology of a simplicial scheme in terms of the sheaf cohomology of its simplicial components. In particular, the group of global sections $\Gamma(Y_{\bullet,}R)=H^0(Y_{\bullet},R)$ is given by the kernel of the morphism $\Gamma(Y_0,R) \rightarrow \Gamma(Y_1,R)$ induced by the simplicial data. This means that
$$H^{0,0}(Y_{\bullet},R)= Ker(H^{0,0}(Y_0,R) \rightarrow H^{0,0}(Y_1,R))$$ 
In other words, $H^{0,0}(Y_{\bullet},R)$ is the free $R$-module with rank equal to the number of connected components of $Y_\bullet$, where the set of connected components of $Y_{\bullet}$ is obtained from the set of connected components of $Y_0$ by identifying all the couples of components of $Y_0$ linked by a connected component of $Y_1$ via the face maps. On the other hand, $H^{0,0}(Y_0,R)$ is the free $R$-module with rank equal to the number of connected components of $Y_0$. It follows that the restriction
$$r_0^*:H^{0,0}(Y_{\bullet},R) \rightarrow H^{0,0}(Y_0,R)$$
is injective, hence $M(p)=M(p_Y) \otimes id_{{\widetilde M}(A)_{Y'_{\bullet}}[1]}$, which is what we aimed to prove. 
\end{proof}

\section{The Nisnevich classifying space}

Let us recall at this point some facts about Nisnevich and \'etale classifying spaces of a linear algebraic group $G$ over $Spec(k)$. 

Denote by $EG$ the simplicial scheme defined by $(EG)_n=G^{n+1}$ with face and degeneracy maps given by partial projections and partial diagonals respectively. There is an obvious action of $G$ on $EG$ induced by the operation in $G$, then the Nisnevich classifying space $BG$ is the simplicial scheme defined by $BG=EG/G$. In other words, $BG$ is the simplicial Nisnevich sheaf with simplicial component $(BG)_n$ given by the Nisnevich sheaf $U \mapsto Hom_{Sm/k}(U,G^n)$ for any $n \geq 0$ and standard face and degeneracy maps of the bar construction. 

Now, consider the morphism of sites $\pi:(Sm/k)_{et} \rightarrow (Sm/k)_{Nis}$. This induces a pair of adjoint functors 
\begin{align*}
{\mathcal H}_s((S&m/k)_{et})\\
\pi^* \uparrow & \downarrow R\pi_*\\
{\mathcal H}_s((S&m/k)_{Nis})
\end{align*}
where $\pi_*$ is the restriction to Nisnevich topology and $\pi^*$ is \'etale sheafification. Then, the \'etale classifying space is defined by $B_{et}G=R\pi_*\pi^*BG$. Furthermore, we recall that in $\cite{MV}$ it is constructed, starting from a faithful representation $\rho:G \hookrightarrow GL(V)$, a geometric model $BG_{gm}$ for the $A^1$-homotopy type of $B_{et}G$, obtained from an infinite-dimensional affine space $\oplus_{i=1}^{\infty} V$ by removing a closed subscheme in order to let the diagonal action of $G$ be free and then taking the quotient.

In this paper we will be mainly interested in Nisnevich classifying spaces. We finish this section by showing some of their features.    

Let $G$ be a linear algebraic group over $k$ and $H$ an algebraic subgroup of $G$. Denote by $\widetilde{B}H$ the bisimplicial scheme $(EH \times EG)/H$, where $H$ acts on $EH \times EG$ diagonally, i.e. $(h_1,\dots,h_m,g_1,\dots,g_n)h=(h_1h,\dots,h_mh,g_1h,\dots,g_nh)$ for any $h_1,\dots,h_m,h$ in $H$ and $g_1,\dots,g_n$ in $G$, and by $\widehat{B}H$ the simplicial scheme $EG/H$. We observe that the natural fibration $\pi:\widehat{B}H \rightarrow BG$ is trivial over simplicial components and has fiber $G/H$. There are two natural maps $\phi:\widetilde{B}H \rightarrow BH$ and $\psi:\widetilde{B}H \rightarrow \widehat{B}H$. We notice that $\phi$ is an isomorphism in ${\mathcal H}_s(k)$ since over each simplicial component it is a trivial fibration with contractible fiber $EG$. On the other hand, $\psi$ is not in general an isomorphism in ${\mathcal H}_s(k)$. However, we have the following statement. 

\begin{prop}\label{BG1}
If the map $Hom_{{\mathcal H}_s(k)}(Spec(R),B_{et}H) \rightarrow Hom_{{\mathcal H}_s(k)}(Spec(R),B_{et}G)$ is injective for any Henselian local ring $R$ over $k$, then $\psi$ is an isomorphism in ${\mathcal H}_s(k)$. In particular, $BH \cong \widehat{B}H$ in ${\mathcal H}_s(k)$.
\end{prop}
\begin{proof} 
	We start by noticing that the restriction of $\psi$ over any simplicial component is given by the morphism $(EH \times G^{n+1})/H \rightarrow G^{n+1}/H$. The simplicial scheme $(EH \times G^{n+1})/H$ is nothing else but the \v{C}ech simplicial scheme $\check{C}({G^{n+1} \rightarrow G^{n+1}/H})$ associated to the $H$-torsor $G^{n+1} \rightarrow G^{n+1}/H$ which becomes split once extended to $G$. In order to check that
	$$\check{C}({G^{n+1} \rightarrow G^{n+1}/H}) \rightarrow G^{n+1}/H$$ 
	is a simplicial weak equivalence it is enough, by \cite[Lemma 1.11]{MV}, to evaluate on henselian local rings. Therefore, we only need to prove that the $H$-torsor $G^{n+1} \rightarrow G^{n+1}/H$ is Nisnevich locally split. Now, the fiber of $G^{n+1} \rightarrow G^{n+1}/H$ over any $Spec(R)$ of $G^{n+1}/H$, where $R$ is henselian local, is given by a $H$-torsor $P \rightarrow Spec(R)$ whose extension to $G$ is split, so split itself by hypothesis. Hence, $G^{n+1} \rightarrow G^{n+1}/H$ is Nisnevich locally split. This implies that $\psi$ is an isomorphism in ${\mathcal H}_s(k)$. 
\end{proof}

In practice, in the unitary group case (as in many other cases), it will be enough to check the hypothesis of the previous proposition only for field extensions of $k$. The reason resides in the fact that rationally trivial hermitian forms are locally trivial (see \cite[Theorem 9.2]{OP}).

There are obvious morphisms $j:BH \rightarrow \widehat{B}H$ and $g:BH \rightarrow BG$ induced by the embedding $H \hookrightarrow G$.

\begin{prop}\label{BG2}
Under the hypothesis of Proposition \ref{BG1}, $j$ is an isomorphism in ${\mathcal H}_s(k)$.
\end{prop}
\begin{proof}
Under the hypothesis of Proposition \ref{BG1} both $\phi$ and $\psi$ are morphisms of bisimplicial schemes which are weak equivalences over simplicial components, hence the induced morphisms on the associated diagonal simplicial schemes $\phi:\Delta(\widetilde{B}H) \rightarrow BH$ and $\psi:\Delta(\widetilde{B}H) \rightarrow \widehat{B}H$ are weak equivalences. In order to complete the proof we only need to notice that the morphisms $j\phi$ and $\psi$ are simplicial homotopic. A simplicial homotopy between them $F_i^{(n)}:(H^{n+1} \times G^{n+1})/H \rightarrow G^{n+2}/H$ is defined for any $n$ and any $0 \leq i \leq n$ by 
$$F_i^{(n)}(h_0,{\dots},h_n,g_0,{\dots},g_n)=(h_0,{\dots},h_i,g_i,{\dots},g_n)$$ 
\end{proof}

It immediately follows from the previous proposition and by noticing that $g=\pi j$ that the morphism $j^*:H(\widehat{B}H) \rightarrow H(BH)$ is an isomorphism of $H(BG)$-modules.

Propositions \ref{BG1} and \ref{BG2} apply in particular to the case when $G$ and $H$ are respectively $O_n$ and $O_{n-1}$. We recall that 
$$A_{q_n} \cong O_n/O_{n-1}$$
where $A_{q_n}$ is the affine quadric defined by the equation $q_n=1$. Moreover, we know that 
$$M(A_{q_n})=\Z \oplus \Z([n/2])[n-1] \in {\mathcal DM}^-_{eff}(k)$$ 
by \cite[Proposition 3.1.3]{SV}. Therefore, Proposition \ref{Thom1} applies to the fibration $\widehat{B}O_{n-1} \rightarrow BO_n$.
 
By previous considerations and by an induction argument, in $\cite{SV}$ it is proven the following theorem.

\begin{thm}\label{SubtleSW}
\cite[Theorem 3.1.1]{SV} There is a unique set $u_1,{\dots},u_n$ of classes in the motivic $\Z/2$-cohomology of $BO_n$ such that $deg(u_i)=([i/2])[i]$, $u_i$ vanishes when restricted to $H(BO_{i-1})$ for any $2 \leq i \leq n$ and
$$H(BO_n)=H[u_1,{\dots},u_n]$$
\end{thm}

These new cohomology classes $u_i$ are called \textit{subtle Stiefel-Whitney classes}.

\section{\v{C}ech simplicial scheme and Rost motive of a quadratic extension}

Let $E=k(\sqrt{\alpha})$ be a quadratic extension of $k$. Then, the motive of $Spec(E)$ in ${\mathcal DM}_{eff}^-(k,\Z/2)$ is the Rost motive $M_{\alpha}$ of the Pfister form $\lva \alpha \rva$. It is proven in \cite{R} that this motive comes endowed with two morphisms $M_{\alpha} \rightarrow T$ and $T \rightarrow M_{\alpha}$ such that the composition $T \rightarrow M_{\alpha} \rightarrow T$ is the $0$ morphism and becomes a split distinguished triangle in ${\mathcal DM}_{eff}^-(E,\Z/2)$.

Moreover, in \cite[Theorem $4.4$]{V2} it is shown that $M_{\alpha}$ can be presented as an extension of two motives of \v{C}ech simplicial schemes. More precisely, in ${\mathcal DM}_{eff}^-(k,\Z/2)$ there is the following distinguished triangle 
\begin{equation}\label{eqn:star}
M_{\alpha} \rightarrow {\mathfrak X}_{\alpha} \rightarrow {\mathfrak X}_{\alpha}[1] \rightarrow M_{\alpha}[1] \tag{$*$}
\end{equation}
where ${\mathfrak X}_{\alpha}$ is the motive of the \v{C}ech simplicial scheme of the Pfister quadric associated to the Pfister form $\lva \alpha \rva$.

Let $\N$ be $Cone(T \rightarrow M_{\alpha})$ and $\n$ be $Cone(M_{\alpha} \rightarrow T)[-1]$. Since $Hom(T,T[j])=0$ for $j \neq 0$ and we are working with $\Z/2$-coefficients, the morphism $T \rightarrow M_{\alpha}$ is uniquely liftable to $\n$ while the morphism $M_{\alpha} \rightarrow T$ is uniquely extendable to $\N$. It immediately follows from the octahedron axiom that $Cone(\N \rightarrow T)[-1] \cong Cone(T \rightarrow \n)$. We will denote this motive by $\M$.

In this section, we will study the above mentioned motives and their motivic cohomology. We start by establishing relations among them.

\begin{prop}\label{12345}
The following isomorphisms hold in ${\mathcal DM}_{eff}^-(k,\Z/2)$:\\
1) $M_{\alpha} \otimes {\mathfrak X}_{\alpha} \cong M_{\alpha}$ via $M_{\alpha} \otimes ({\mathfrak X}_{\alpha} \rightarrow T)$;\\
2) $\N \otimes {\mathfrak X}_{\alpha} \cong {\mathfrak X}_{\alpha}$ via $(\N \rightarrow T) \otimes {\mathfrak X}_{\alpha}$;\\
3) $M_{\alpha} \otimes \N \cong M_{\alpha}$ via $M_{\alpha} \otimes (\N \rightarrow T)$;\\
4) $\n \otimes \N \cong T$;\\
5) $\M \otimes \N \cong \M[1]$ via $\M \otimes (\N \rightarrow T[1])$.
\end{prop}
\begin{proof}
1) Since ${\mathfrak X}_{\alpha}$ is a projector in ${\mathcal DM}_{eff}^-(k,\Z/2)$ we have that ${\mathfrak X}_{\alpha} \otimes {\mathfrak X}_{\alpha} \cong {\mathfrak X}_{\alpha}$. Hence, by tensoring with ${\mathfrak X}_{\alpha}$ the distinguished triangle 
$$M_{\alpha} \rightarrow {\mathfrak X}_{\alpha} \rightarrow {\mathfrak X}_{\alpha}[1] \rightarrow M_{\alpha}[1]$$
we obtain that $M_{\alpha} \otimes {\mathfrak X}_{\alpha} \cong M_{\alpha}$.

2) Therefore, by tensoring with ${\mathfrak X}_{\alpha}$ the distinguished triangle
$$T \rightarrow M_{\alpha} \rightarrow \N \rightarrow T[1]$$
and by recalling that ${\mathfrak X}_{\alpha} \rightarrow M_{\alpha}$ from \eqref{eqn:star} factors through $T$ we get that $\N \otimes {\mathfrak X}_{\alpha} \cong {\mathfrak X}_{\alpha}$. 

3) It follows formally from 1) and 2).

4) On the other hand, by tensoring with $\N$ the distinguished triangle
$$\n \rightarrow M_{\alpha} \rightarrow T \rightarrow \n[1]$$
and by noticing that $(M_{\alpha} \rightarrow T) \otimes \N$ coincides with $M_{\alpha} \rightarrow \N$ we obtain that $\n \otimes \N \cong T$.

5) Finally, by tensoring with $\N$ the distinguished triangle
$$T \rightarrow \n \rightarrow \M \rightarrow T[1]$$
and by noticing that $(T \rightarrow \n) \otimes \N$ coincides with $\N \rightarrow T$ we have that $\M \otimes \N \cong \M[1]$. 
\end{proof}

From the previous proposition we immediately deduce the following lemma.

\begin{lem}\label{Na}
In ${\mathcal DM}_{eff}^-(k,\Z/2)$ for any $n \in \mathbb{N}$ there are the following distinguished triangles:\\
1) $\Nnm \rightarrow M_{\alpha} \rightarrow \Nn \rightarrow \Nnm[1]$;\\
2) $\M[n-1] \rightarrow \Nn \rightarrow \Nnm \rightarrow \M[n]$.\\
Here, $M_{\alpha} \rightarrow \Nn$ and $\Nn \rightarrow \Nnm$ are the unique non-zero morphisms between the respective objects.
\end{lem}
\begin{proof}
1) It follows immediately from $3)$ of Proposition \ref{12345} by tensoring the distinguished triangle
$$T \rightarrow M_{\alpha} \rightarrow \N \rightarrow T[1]$$
with the appropriate power of $\N$.

2) It follows immediately from $5)$ of Proposition \ref{12345} by tensoring the distinguished triangle
$$\M \rightarrow \N \rightarrow T \rightarrow \M[1]$$
with the appropriate power of $\N$.  
\end{proof}

At this point, we present the motivic cohomology of $\M$, which will be used in the main result of this section, namely the computation of the motivic cohomology of tensor powers of $\N$.

\begin{lem}\label{HNat}
There exists a cohomology class $\mu$ of bidegree $(0)[1]$ such that the motivic cohomology of $\M$ is given by
$$H(\M)={\frac {K^M(k)/2} {Ann(\{\alpha\})}} \cdot \mu$$
So, the motivic cohomology of $\M$ is concentrated on a single diagonal. 
\end{lem}
\begin{proof}
After applying the octahedron axiom twice to the distinguished triangle
$$M_{\alpha} \rightarrow {\mathfrak X}_{\alpha} \rightarrow {\mathfrak X}_{\alpha}[1] \rightarrow M_{\alpha}[1]$$
we get the distinguished triangle
$$\M[-1] \rightarrow \widetilde{{\mathfrak X}_{\alpha}}[1] \rightarrow \widetilde{{\mathfrak X}_{\alpha}} \rightarrow \M$$
where $\widetilde{{\mathfrak X}_{\alpha}}$ is $Cone({\mathfrak X}_{\alpha} \rightarrow T)[-1]$.

The motivic cohomology of $\widetilde{{\mathfrak X}_{\alpha}}$ has been computed in the original version of \cite{OVV} and \cite{Y}. It is described by
$$H(\widetilde{{\mathfrak X}_{\alpha}})=\Z/2[\mu] \cdot \mu \otimes {\frac {K^M(k)/2} {Ann(\{\alpha\})}}$$
Therefore, by the long exact sequence in motivic cohomology induced by the previous distinguished triangle and by recalling that $H^{*,*'}(\widetilde{{\mathfrak X}_{\alpha}}) \rightarrow H^{*-1,*'}(\widetilde{{\mathfrak X}_{\alpha}})$ sends $\mu^j$ to $\mu^{j-1}$ we get the description of $H(\M)$. 
\end{proof}

We are now ready to compute the motivic cohomology of any tensor power of $\N$. This result will be essential in the next section for the proof of the main result.

\begin{prop}\label{HNa}
For any $n \in \mathbb{N}$ there exist cohomology classes $\mu_i$ of bidegree $(0)[i]$ for $1 \leq i \leq n$ such that the motivic cohomology of the $\emph{n}$th tensor power of $\N$ as an $H$-module is given by
$$H(\Nn)=H \oplus \bigoplus\limits_{i=1}^n {\frac {K^M(k)/2} {Ann(\{\alpha\})}} \cdot \mu_i$$
where the $H$-module structure is described by the relations $\tau \mu_i=\{\alpha\}\mu_{i-1}$ ($\mu_0=1$ by convention).
\end{prop}
\begin{proof}
We will proceed by induction on $n$. For $n=1$ the distinguished triangle 
$$\M \rightarrow \N \rightarrow T \rightarrow \M[1]$$
induces the following long exact sequence in motivic cohomology
$${\dots} \rightarrow H^{*-1,*'}(\M) \rightarrow H^{*,*'} \rightarrow H^{*,*'}(\N) \rightarrow H^{*,*'}(\M) \rightarrow {\dots}$$
From Lemma \ref{HNat} it follows that 
$$H^{*,*'}(\N)=
\begin{cases}
H^{*,*'}(\M), & *>*' \\
H^{*,*'}, & * \leq *'
\end{cases}
$$
which implies that
$$H(\N)=H \oplus {\frac {K^M(k)/2} {Ann(\{\alpha\})}} \cdot \mu$$
On the other hand, after tensoring with ${\mathfrak X}_{\alpha}$ the distinguished triangle 
$$T \rightarrow M_{\alpha} \rightarrow \N \rightarrow T[1]$$
we get a morphism of long exact sequences in motivic cohomology
$$
\xymatrix{
{\dots} \ar@{->}[r] & H^{*-1,*'}(M_{\alpha}) \ar@{->}[r] \ar@{->}[d] & H^{*-1,*'} \ar@{->}[r] \ar@{->}[d] & H^{*,*'}(\N) \ar@{->}[r] \ar@{->}[d] & H^{*,*'}(M_{\alpha}) \ar@{->}[r] \ar@{->}[d] & {\dots}\\
{\dots} \ar@{->}[r] & H^{*-1,*'}(M_{\alpha}) \ar@{->}[r] & H^{*-1,*'}({\mathfrak X}_{\alpha}) \ar@{->}[r] & H^{*,*'}({\mathfrak X}_{\alpha}) \ar@{->}[r] & H^{*,*'}(M_{\alpha}) \ar@{->}[r] & {\dots}
}
$$
By a four lemma argument we deduce that $H(\N) \rightarrow H({\mathfrak X}_{\alpha})$ is injective. Therefore, $\tau \mu=\{\alpha\}$ in $H(\N)$, since the same relation holds in $H({\mathfrak X}_{\alpha})$. That completes the induction basis. 

Now, suppose the statement holds for $n-1$. Then, by $2)$ of Lemma \ref{Na} we have the following long exact sequence in motivic cohomology
$${\dots} \rightarrow H^{*-n,*'}(\M) \rightarrow H^{*,*'}(\Nnm) \rightarrow H^{*,*'}(\Nn) \rightarrow H^{*-n+1,*'}(\M) \rightarrow {\dots}$$
From Lemma \ref{HNat} and by induction hypothesis we have
$$H^{*,*'}(\Nn)=
\begin{cases}
H^{*-n+1,*'}(\M), & *>*'+n-1 \\
H^{*,*'}(\Nnm), & * \leq *'+n-1
\end{cases}
$$
which implies that there exists $\mu_n$ in bidegree $(0)[n]$ such that
$$H(\Nn)=H(\Nnm) \oplus {\frac {K^M(k)/2} {Ann(\{\alpha\})}}\cdot \mu_n=H \oplus \bigoplus\limits_{i=1}^n {\frac {K^M(k)/2} {Ann(\{\alpha\})}}\cdot \mu_i$$
From $1)$ of Lemma \ref{Na} we have the following long exact sequence in motivic cohomology
$${\dots} \rightarrow H^{*-1,*'}(\Nnm) \rightarrow H^{*,*'}(\Nn) \rightarrow H^{*,*'}(M_{\alpha}) \rightarrow H^{*,*'}(\Nnm) \rightarrow {\dots}$$
that maps $\mu_{i-1} \in H^{i-1,0}(\Nnm)$ to $\mu_i \in H^{i,0}(\Nn)$ since $H^{i,0}(M_{\alpha})=0$ for $i > 0$. Hence, by induction hypothesis, $\tau \mu_i=\{\alpha\}\mu_{i-1}$ in $H(\Nn)$ and the proof is complete.  
\end{proof}

By $2)$ of Lemma \ref{12345} there is a chain of morphisms 
$${\mathfrak X}_{\alpha} \rightarrow {\dots} \rightarrow \Nn \rightarrow \Nnm \rightarrow {\dots} \rightarrow \N \rightarrow T$$ 
that induces in cohomology the chain of homomorphisms 
$$H \rightarrow H(\N) \rightarrow {\dots} \rightarrow H(\Nnm) \rightarrow H(\Nn) \rightarrow {\dots} \rightarrow H({\mathfrak X}_{\alpha})$$
which sends $\mu \in H(\N)$ to $\mu \in H({\mathfrak X}_{\alpha})$.

We now highlight an interesting relation between the invertible motive $\N$ and the projector ${\mathfrak X}_{\alpha}$.

\begin{prop}\label{RS}
The homomorphisms $H(\Nn) \rightarrow H({\mathfrak X}_{\alpha})$ are injective for all $n \in \mathbb{N}$. Moreover, $H({\mathfrak X}_{\alpha})=\varinjlim H(\Nn)$.
\end{prop}
\begin{proof}
We have already noticed that $H(\N) \rightarrow H({\mathfrak X}_{\alpha})$ is injective and  maps $\mu = \mu_1$ to $\mu$. Now, suppose by induction that the homomorphism $H(\Nnm) \rightarrow H({\mathfrak X}_{\alpha})$ is injective. Notice that there is a commutative diagram
$$
\xymatrix{
H(\Nnm) \otimes H(\N) \ar@{->}[r] \ar@{->}[d] & H(\Nn) \ar@{->}[d]\\
H({\mathfrak X}_{\alpha}) \otimes H({\mathfrak X}_{\alpha}) \ar@{->}[r]^(0.6){\smile} & H({\mathfrak X}_{\alpha})
}
$$
where the bottom horizontal map is the usual cup product in $H({\mathfrak X}_{\alpha})$. It follows that the right vertical map sends $\mu_i$ to $\mu^i$ for any $i \leq n$. This completes the proof. 
\end{proof}

Later on we will need also the following description of the motivic cohomology of $\n$.

\begin{lem}\label{HMat}
The motivic cohomology of $\n$ is given by
$$H(\n)=Ann(\{\alpha\}) \oplus H\cdot \tau$$
\end{lem}
\begin{proof}
After applying the octahedron axiom to the distinguished triangle
$$M_{\alpha} \rightarrow {\mathfrak X}_{\alpha} \rightarrow {\mathfrak X}_{\alpha}[1] \rightarrow M_{\alpha}[1]$$
we obtain 
$${\mathfrak X}_{\alpha} \rightarrow \n \rightarrow \widetilde{{\mathfrak X}_{\alpha}} \rightarrow {\mathfrak X}_{\alpha}[1]$$
which induces in motivic cohomology the following long exact sequence
$${\dots} \rightarrow H^{*-1,*'}({\mathfrak X}_{\alpha}) \rightarrow H^{*,*'}(\widetilde{{\mathfrak X}_{\alpha}}) \rightarrow H^{*,*'}(\n) \rightarrow H^{*,*'}({\mathfrak X}_{\alpha}) \rightarrow {\dots}$$
Hence, the result follows by noticing that $H^{*,*'}(\widetilde{{\mathfrak X}_{\alpha}})$ is the $* > *'$ part of $H^{*,*'}({\mathfrak X}_{\alpha})$, while $H^{*,*'}$ is the $* \leq *'$ part of it, and that $H^{*-1,*'}({\mathfrak X}_{\alpha}) \rightarrow H^{*,*'}(\widetilde{{\mathfrak X}_{\alpha}})$ sends $\mu^{i-1}$ to $\mu^i$. 
\end{proof}

\section{The motivic cohomology ring of $BU_n(E/k)$} 

Our goal in this section is to compute by using the techniques presented in section $3$ and $4$ the motivic cohomology of the Nisnevich classifying space of $U_n(E/k)$, the unitary group associated to the standard split hermitian form $h_n$ of the extension $E/k$.

At first, let us show some preliminary results which will be useful in the proof of the main theorem. 

\begin{prop}
The homogeneous variety $U_n(E/k)/U_{n-1}(E/k)$ is isomorphic to the affine quadric $A_{h_n}$ defined by the equation $\widetilde{h}_n=1$.
\end{prop}
\begin{proof}
Let $V$ be an $n$-dimensional $E$-vector space and let $A_{h_n}$ be the subset of $V$ defined by the equation $h_n=1$. Then, $V$ can be considered as a $2n$-dimensional $k$-vector space in which $A_{h_n}$ is the affine quadric defined by the equation $\widetilde{h}_n=1$. The action of $U_n(E/k)$ on $A_{h_n}$ is transitive since for any two vectors of $A_{h_n}$ there exists the product of at most two reflections which sends one to the other (see the proof of \cite[Theorem 9.5]{S}). Moreover, the isotropy group of the vector $(1,0,\dots,0)$ is isomorphic to $U_{n-1}(E/k)$. This implies the desired result. 
\end{proof}

At this point, in order to apply Proposition \ref{Thom1} to the unitary case, we need to study the motive of the affine quadric $A_{h_n}$. 

\begin{prop}\label{MAQ}
The motive in ${\mathcal DM}_{eff}^-(k,\Z/2)$ of the affine quadric $A_{h_n}$ is given by
$$M(A_{h_n})=
\begin{cases}
T \oplus \N(n)[2n-1], & n \: odd \\
T \oplus T(n)[2n-1], & n \: even
\end{cases}
$$
\end{prop}
\begin{proof}
We start by noticing that the quadratic form 
$$\widetilde{h}_n=
\begin{cases}
\lva \alpha \rva \perp (n-1)\HF, & n \: odd \\
n\HF, & n \: even
\end{cases}
$$
For a quadratic form $q$ let us denote by $Q$ the projective quadric defined by $q=0$, by $Q'$ the projective quadric defined by $q=z^2$ and by $A$ the affine quadric defined by $q=1$. Then, we have in ${\mathcal DM}_{eff}^-(k,\Z/2)$ the following Gysin triangle
$$M(A) \rightarrow M(Q') \rightarrow M(Q)(1)[2] \rightarrow M(A)[1]$$

In the case $q=n\HF$ the previous triangle becomes
$$M(A) \rightarrow \bigoplus\limits_{i=0}^{2n-1} T(i)[2i] \rightarrow \bigoplus\limits_{i=1}^{2n-1} T(i)[2i] \oplus T(n)[2n] \rightarrow M(A)[1]$$
which implies that, for $n$ even, $M(A_{h_n})=M(A)=T \oplus T(n)[2n-1]$.

In the case $q=\lva \alpha \rva$ we have
$$M(A) \rightarrow T \oplus T(1)[2] \rightarrow M_{\alpha}(1)[2] \rightarrow M(A)[1]$$
from which it follows that $M(A_{h_1})=M(A)=T \oplus \N(1)[1]$.

The general case $n$ odd follows from \cite[Lemma 34]{B} . Namely, we have
$$\widetilde{M}(A_{h_n})=\widetilde{M}(A_{h_1})(n-1)[2n-2]=\N(n)[2n-1]$$
that implies $M(A_{h_n})=T \oplus \N(n)[2n-1]$. 
\end{proof}

Before going ahead with the main theorem of this section, we notice that $U_n(E/k)$-torsors over $Spec(k)$ are in one-to-one correspondence with hermitian forms associated to the quadratic extension $E/k$ or, which is the same, with quadratic forms over $k$ divisible by $\lva \alpha \rva$. The map from the set of $U_{n-1}(E/k)$-torsors to the set of $U_n(E/k)$-torsors sends a hermitian form $h$ to $h \perp \la (-1)^{n-1} \ra$ or, analogously, a quadratic form $q$ divisible by $\lva \alpha \rva$ to $q \perp (-1)^{n-1} \lva \alpha \rva$. Since Witt cancellation holds for quadratic forms, the previous remark assures that $U_{n-1}(E/k)$-torsors inject in $U_n(E/k)$-torsors over any field extension of $k$, which allows us to use Propositions \ref{BG1} and \ref{BG2} in the unitary case.

\begin{thm}\label{main}
For any $m,n \in \Z_{\geq 0}$ there exist cohomology classes $c_i$ of bidegree $(i)[2i]$ for $1 \leq i \leq n$ such that the motivic cohomology of ${\mathfrak X}_{\alpha} \otimes BU_n(E/k)$ and $\Nm \otimes BU_n(E/k)$ is described respectively by
$$H({\mathfrak X}_{\alpha} \otimes BU_n(E/k))=H({\mathfrak X}_{\alpha})[c_1,{\dots},c_n]$$
and
$$H(\Nm \otimes BU_n(E/k))=\bigoplus\limits_{i_1,{\dots},i_n \in \Z_{\geq 0}} H(\Nodd)\cdot c_1^{i_1}{\cdots}c_n^{i_n}$$
where the obvious homomorphisms of $H(BU_n(E/k))$-modules 
$$H(\Nm \otimes BU_n(E/k)) \rightarrow H({\mathfrak X}_{\alpha} \otimes BU_n(E/k))$$ 
are injective. Moreover, $H({\mathfrak X}_{\alpha} \otimes BU_n(E/k))=\varinjlim H(\Nm \otimes BU_n(E/k))$.
\end{thm}
\begin{proof}
We will proceed by induction on $n$. The induction basis follows immediately from noticing that $BU_0(E/k) \cong Spec(k)$ and by Proposition \ref{RS}.

Now, suppose the result holds for $n-1$. Then, since $\N \otimes {\mathfrak X}_{\alpha} \cong {\mathfrak X}_{\alpha}$ by $2)$ of Proposition \ref{12345} and applying Propositions \ref{Thom1}, \ref{BG1}, \ref{BG2} and \ref{MAQ} to the coherent morphism ${\mathfrak X}_{\alpha} \otimes \widehat{B}U_{n-1}(E/k) \rightarrow {\mathfrak X}_{\alpha} \otimes BU_n(E/k)$, we obtain the following long exact sequence in motivic cohomology
$${\dots} \rightarrow H^{*-1,*'}({\mathfrak X}_{\alpha} \otimes BU_{n-1}(E/k)) \xrightarrow{h^*} H^{*-2n,*'-n}({\mathfrak X}_{\alpha} \otimes BU_n(E/k)) \xrightarrow{f^*}$$
$$H^{*,*'}({\mathfrak X}_{\alpha} \otimes BU_n(E/k)) \xrightarrow{g^*} H^{*,*'}({\mathfrak X}_{\alpha} \otimes BU_{n-1}(E/k)) \rightarrow {\dots}$$
Note that, even after having replaced $H({\mathfrak X}_{\alpha} \otimes \widehat{B}U_{n-1}(E/k))$ with $H({\mathfrak X}_{\alpha} \otimes BU_{n-1}(E/k))$, this stays a sequence of $H({\mathfrak X}_{\alpha} \otimes BU_n(E/k))$-modules by the remark just after Proposition \ref{BG2}. By induction hypothesis, $H({\mathfrak X}_{\alpha} \otimes BU_{n-1}(E/k)$) is freely generated as an $H({\mathfrak X}_{\alpha})$-algebra by $c_1,\dots,c_{n-1}$ which are all uniquely liftable to $H({\mathfrak X}_{\alpha} \otimes BU_n(E/k))$, since ${\mathfrak X}_{\alpha} \otimes BU_n(E/k)$ is the motive of a smooth simplicial scheme and, so, has no cohomology in negative round degrees. Hence, $g^*$ is an epimorphism as it is a ring homomorphism, $h^*$ is trivial and $f^*$ is a monomorphism. Denoting by $c_n$ the element $f^*(1)$ we obtain the result
$$H({\mathfrak X}_{\alpha} \otimes BU_n(E/k))=H({\mathfrak X}_{\alpha})[c_1,{\dots},c_n]$$

For the rest of the induction step we will consider separately two cases.

1) \textit{n even}: for any $m \in \mathbb{N}$ we have the following long exact sequence in motivic cohomology of $H(BU_n(E/k))$-modules
$${\dots} \rightarrow H^{*-1,*'}(\Nm \otimes BU_{n-1}(E/k)) \xrightarrow{h^*} H^{*-2n,*'-n}(\Nm \otimes BU_n(E/k)) \xrightarrow{f^*}$$
$$H^{*,*'}(\Nm \otimes BU_n(E/k)) \xrightarrow{g^*} H^{*,*'}(\Nm \otimes BU_{n-1}(E/k)) \rightarrow {\dots}$$
For $m=0$, by induction hypothesis, $H(BU_{n-1}(E/k))$ is generated as an $H$-algebra by $c_1,{\dots},c_{n-1}$ and $\mu c_l$ for any odd $l<n$. By degree reasons these cohomology classes are all uniquely liftable to $H(BU_n(E/k))$. Therefore, $g^*$ is an epimorphism since it is a ring homomorphism. This assures that, for any $m$, $H(\Nm \otimes BU_{n-1}(E/k))$ is generated as an $H(BU_{n}(E/k))$-module by $\mu_i$ for all $i \leq m$. By degree reasons the $\mu_i$ are all uniquely liftable to $H(\Nm \otimes BU_n(E/k))$. Now $g^*$ happens to be surjective since it is a homomorphism of $H(BU_{n}(E/k))$-modules. Hence, $h^*$ is the $0$ homomorphism and $f^*$ is a monomorphism. Then, denoting by $c_n$ the cohomology class $f^*(1)$ we have, for any $m$, the following morphism of short exact sequences of $H(BU_n(E/k))$-modules
$$
\xymatrixcolsep{1.3pc}\xymatrix{
0 \ar@{->}[r] & H^{*-2n,*'-n}(\Nm \otimes BU_n(E/k)) \ar@{->}[r]^(0.54){\cdot c_n} \ar@{->}[d] & H^{*,*'}(\Nm \otimes BU_n(E/k)) \ar@{->}[r] \ar@{->}[d] & H^{*,*'}(\Nm \otimes BU_{n-1}(E/k)) \ar@{->}[r] \ar@{^{(}->}[d] &0\\
0 \ar@{->}[r] & H^{*-2n,*'-n}({\mathfrak X}_{\alpha} \otimes BU_n(E/k)) \ar@{->}[r]^(0.54){\cdot c_n} & H^{*,*'}({\mathfrak X}_{\alpha} \otimes BU_n(E/k)) \ar@{->}[r] & H^{*,*'}({\mathfrak X}_{\alpha} \otimes BU_{n-1}(E/k)) \ar@{->}[r] & 0
}
$$
By induction on square degree and by a standard four lemma argument, the central vertical morphism is injective. Moreover, by an induction argument on square degree and looking at the previous upper short exact sequence we get that
$$H(\Nm \otimes BU_{n}(E/k))=\bigoplus\limits_{i \in \Z_{\geq 0}}H(\Nm \otimes BU_{n-1}(E/k))\cdot c_n^i=\bigoplus\limits_{i_1,{\dots},i_n \in \Z_{\geq 0}} H(\Nodd)\cdot c_1^{i_1}{\cdots}c_n^{i_n}$$
as an $H(BU_n(E/k))$-submodule of $H({\mathfrak X}_{\alpha} \otimes BU_n(E/k))$.

2) \textit{n odd}: as before for any $m$ we have the following long exact sequence in motivic cohomology of $H(BU_n(E/k))$-modules
$${\dots} \rightarrow H^{*-1,*'}(\Nm \otimes BU_{n-1}(E/k)) \xrightarrow{h^*} H^{*-2n,*'-n}(\Nmp \otimes BU_n(E/k)) \xrightarrow{f^*}$$
$$H^{*,*'}(\Nm \otimes BU_n(E/k)) \xrightarrow{g^*} H^{*,*'}(\Nm \otimes BU_{n-1}(E/k)) \rightarrow {\dots}$$
As in the previous case, for $m=0$ the induction hypothesis implies that $H(BU_{n-1}(E/k))$ is generated as an $H$-algebra by $c_1,{\dots},c_{n-1}$ and $\mu c_l$ for any odd $l<n$. By the same degree reasons they are all uniquely liftable to $H(BU_n(E/k))$. Thus, $g^*$ is an epimorphism since it is a ring homomorphism. This is enough to show that, for any $m$, $H(\Nm \otimes BU_{n-1}(E/k))$ is generated as an $H(BU_{n}(E/k))$-module by $\mu_i$ for all $i \leq m$. Again the $\mu_i$ are uniquely liftable to $H(\Nm \otimes BU_n(E/k))$. It follows that $g^*$ is surjective, $h^*$ is trivial and $f^*$ is injective. Then, denoting by $c_n$ the cohomology class $f^*(1)$ we have, for any $m$, the following morphism of short exact sequences of $H(BU_n(E/k))$-modules
$$
\xymatrixcolsep{1.1pc}\xymatrix{
0 \ar@{->}[r] & H^{*-2n,*'-n}(\Nmp \otimes BU_n(E/k)) \ar@{->}[r]^(0.54){\cdot c_n} \ar@{->}[d] & H^{*,*'}(\Nm \otimes BU_n(E/k)) \ar@{->}[r] \ar@{->}[d] & H^{*,*'}(\Nm \otimes BU_{n-1}(E/k)) \ar@{->}[r] \ar@{^{(}->}[d] &0\\
0 \ar@{->}[r] & H^{*-2n,*'-n}({\mathfrak X}_{\alpha} \otimes BU_n(E/k)) \ar@{->}[r]^(0.54){\cdot c_n} & H^{*,*'}({\mathfrak X}_{\alpha} \otimes BU_n(E/k)) \ar@{->}[r] & H^{*,*'}({\mathfrak X}_{\alpha} \otimes BU_{n-1}(E/k)) \ar@{->}[r] & 0
}
$$
By the very same arguments of the previous case, the central vertical morphism is injective and
$$H(\Nm \otimes BU_{n}(E/k))=\bigoplus\limits_{i \in \Z_{\geq 0}}H(\Nmi \otimes BU_{n-1}(E/k))\cdot c_n^i=\bigoplus\limits_{i_1,{\dots},i_n \in \Z_{\geq 0}} H(\Nodd)\cdot c_1^{i_1}{\cdots}c_n^{i_n}$$
as an $H(BU_n(E/k))$-submodule of $H({\mathfrak X}_{\alpha} \otimes BU_n(E/k))$, which completes the proof. 
\end{proof}

As a corollary of the previous theorem we obtain the description of the motivic cohomology ring of $BU_n(E/k)$ as an $H$-algebra.

\begin{thm}\label{BUn}
For any $n \in \Z_{\geq 0}$ there exist cohomology classes $c_i$ of bidegree $(i)[2i]$ for $1 \leq i \leq n$ and $d_j$ of bidegree $(j)[2j+1]$ for $1 \leq j \: odd \leq n$ such that the motivic cohomology ring of $BU_n(E/k)$ is given by 
$$H(BU_n(E/k))={\frac {H[c_i,d_j]_{1 \leq i \leq n,1 \leq j \: odd \leq n}} {R}}$$
where $R$ is the ideal generated by $\tau d_j+\{\alpha\} c_j$, $Ann(\{\alpha\})\cdot d_j$ and $c_{j'}d_j+c_jd_{j'}$ for any $1 \leq j,j' \: odd \leq n$.
\end{thm}
\begin{proof}
By Theorem \ref{main} we have a monomorphism of rings
$$H(BU_n(E/k))=\bigoplus\limits_{i_1,{\dots},i_n \in \Z_{\geq 0}} H(\nodd)\cdot c_1^{i_1}{\cdots}c_n^{i_n} \rightarrow H({\mathfrak X}_{\alpha})[c_1,{\dots},c_n]$$
from which we deduce that $H(BU_n(E/k))$ is generated as an $H$-algebra by the $c_i$ and the $\mu c_j$ for $j$ odd. Let us denote by $d_j$ these elements. Then, the relations among $c_i$ and $d_j$ that generate $R$ follow immediately by Proposition \ref{HNa} and by noticing that $\mu c_j \cdot c_{j'}=c_j \cdot \mu c_{j'}$ for any $1 \leq j,j' \: odd \leq n$. This amounts to say that there is an epimorphism 
$$p:{\frac {H[c_i,d_j]_{1 \leq i \leq n,1 \leq j \: odd \leq n}} {R}} \rightarrow H(BU_n(E/k))$$ 
We can check its injectivity by looking separately at each restriction 
$$p:p^{-1}(H(\nodd) \cdot c_1^{i_1}{\cdots}c_n^{i_n}) \rightarrow H(\nodd) \cdot c_1^{i_1}{\cdots}c_n^{i_n}$$
Notice that $H(\nodd) \cdot c_1^{i_1}{\cdots}c_n^{i_n}$ is generated as a $K^M(k)/2$-module by $\mu^m c_1^{i_1}{\cdots}c_n^{i_n}$ for any $0<m \leq \sum_{l \: odd} i_l$ and $\tau^{m'} c_1^{i_1}{\cdots}c_n^{i_n}$ for any $m' \geq 0$. Moreover, the elements in $p^{-1}(H(\nodd) \cdot c_1^{i_1}{\cdots}c_n^{i_n})$ that map to these generators through $p$ are unique. Then, injectivity follows by looking at the restriction of $p$ on each diagonal of $p^{-1}(H(\nodd) \cdot c_1^{i_1}{\cdots}c_n^{i_n})$ which is an isomorphism to ${\frac {K^M(k)/2} {Ann(\{\alpha\})}}$ on positive diagonals and to $K^M(k)/2$ on the others. 
\end{proof}

\section{Comparison between $BU_n(E/k)$ and $BO(\widetilde{h}_n)$}

Since there exists an obvious homomorphism of groups $U_n(E/k) \rightarrow O(\widetilde{h}_n)$, it is reasonable to compare the classifying spaces $BU_n(E/k)$ and $BO(\widetilde{h}_n)$ and, in particular, the characteristic classes arising from both.

Before proceeding, we highlight that, given a quadratic form $q$, there is the following isomorphism
$$O(q \perp \la b \ra)/O(q) \cong A_{q \perp \la b \ra=b}$$
where by $A_{q \perp \la b \ra=b}$ we mean the affine quadric defined by the equation $q \perp \la b \ra=b$.

For sake of simplicity, we will express by $p_n$ the quadratic form $\lva \alpha \rva \perp (n-1)\HF$ and by $p_{n-{\frac 1 2}}$ the quadratic form $\la -\alpha \ra \perp (n-1)\HF$.

In the following theorem we compute the motivic cohomology ring of $BO(p_n)$.

\begin{thm}\label{Hpn}
For any $n \in \Z_{\geq 0}$ there exist cohomology classes $u_i$ of bidegree $([i/2])[i]$ for $1 \leq i \leq 2n$ and a class $v_{2n+1}$ of bidegree $(n)[2n+1]$ such that the motivic cohomology ring of $BO(p_n)$ is given by 
$$H(BO(p_n))=\frac{H[u_1,{\dots},u_{2n},v_{2n+1}]}{(\tau v_{2n+1}+\{\alpha\} u_{2n},Ann(\{\alpha\})\cdot v_{2n+1})}$$
\end{thm}
\begin{proof}
We start by noticing that
$$O(p_n)/O(p_{n-{\frac 1 2}}) \cong A_{p_n=1}$$
From the fact that $O(p_{n-{\frac 1 2}}) \cong O_{2n-1}$ we obtain by Theorem \ref{SubtleSW} that
$$H(BO(p_{n-{\frac 1 2}}))=H[u_1,\dots,u_{2n-1}]$$
Then, from \cite[Proposition $3.2.4$]{SV} it follows that
$$H(\Nm \otimes BO(p_{n-{\frac 1 2}}))=H(\Nm) \otimes_H H[u_1,\dots,u_{2n-1}]$$ 
and
$$H({\mathfrak X}_{\alpha} \otimes BO(p_{n-{\frac 1 2}}))=H({\mathfrak X}_{\alpha})[u_1,\dots,u_{2n-1}]$$
Now, by recalling that $M(A_{p_n=1})=T \oplus \N(n)[2n-1]$ and $\N \otimes {\mathfrak X}_{\alpha} \cong {\mathfrak X}_{\alpha}$ and using Propositions \ref{Thom1}, \ref{BG1} and \ref{BG2}, we obtain a long exact sequence in motivic cohomology
$${\dots} \rightarrow H^{*-1,*'}({\mathfrak X}_{\alpha} \otimes BO(p_{n-{\frac 1 2}})) \xrightarrow{h^*} H^{*-2n,*'-n}({\mathfrak X}_{\alpha} \otimes BO(p_n)) \xrightarrow{f^*}$$
$$H^{*,*'}({\mathfrak X}_{\alpha} \otimes BO(p_n)) \xrightarrow{g^*} H^{*,*'}({\mathfrak X}_{\alpha} \otimes BO(p_{n-{\frac 1 2}})) \rightarrow {\dots}$$
Hence, by the same arguments of Theorem \ref{main} and denoting by $u_{2n}$ the class $f^*(1)$ we get that
$$H({\mathfrak X}_{\alpha} \otimes BO(p_n))=H({\mathfrak X}_{\alpha})[u_1,\dots,u_{2n}]$$
As in the odd case of Theorem \ref{main}, for any $m$ we get a long exact sequence of $H(BO(p_n))$-modules
$${\dots} \rightarrow H^{*-1,*'}(\Nm \otimes BO(p_{n-{\frac 1 2}})) \xrightarrow{h^*} H^{*-2n,*'-n}(\Nmp \otimes BO(p_n)) \xrightarrow{f^*}$$
$$H^{*,*'}(\Nm \otimes BO(p_n)) \xrightarrow{g^*} H^{*,*'}(\Nm \otimes BO(p_{n-{\frac 1 2}})) \rightarrow {\dots}$$
Hence, by exactly the same arguments of Theorem \ref{main} and denoting by $u_{2n}$ the class $f^*(1)$ we obtain, for any $m$, a morphism of short exact sequences of $H(BO(p_n))$-modules
$$
\xymatrixcolsep{1.1pc}\xymatrix{
	0 \ar@{->}[r] & H^{*-2n,*'-n}(\Nmp \otimes BO(p_n)) \ar@{->}[r]^(0.56){\cdot u_{2n}} \ar@{->}[d] & H^{*,*'}(\Nm \otimes BO(p_n)) \ar@{->}[r] \ar@{->}[d] & H^{*,*'}(\Nm \otimes BO(p_{n-{\frac 1 2}})) \ar@{->}[r] \ar@{^{(}->}[d] &0\\
	0 \ar@{->}[r] & H^{*-2n,*'-n}({\mathfrak X}_{\alpha} \otimes BO(p_n)) \ar@{->}[r]^(0.54){\cdot u_{2n}} & H^{*,*'}({\mathfrak X}_{\alpha} \otimes BO(p_n)) \ar@{->}[r] & H^{*,*'}({\mathfrak X}_{\alpha} \otimes BO(p_{n-{\frac 1 2}})) \ar@{->}[r] & 0
}
$$
From this it follows that
$$H(\Nm \otimes BO(p_n))=\bigoplus\limits_{i \in \Z_{\geq 0}}H(\Nmi \otimes BO(p_{n-{\frac 1 2}}))\cdot u_{2n}^i=\bigoplus\limits_{i_1,{\dots},i_{2n} \in \Z_{\geq 0}} H(\Nmin)\cdot u_1^{i_1}{\cdots}u_{2n}^{i_{2n}}$$
and, setting $m=0$, we obtain
$$H(BO(p_n))=\bigoplus\limits_{i_1,{\dots},i_{2n} \in \Z_{\geq 0}} H(\Ni) \cdot u_1^{i_1}{\cdots}u_{2n}^{i_{2n}}$$
Moreover, we have a monomorphism of $H$-algebras
$$H(BO(p_n)) \rightarrow H({\mathfrak X}_{\alpha} \otimes BO(p_n))$$
from which we deduce, as in Theorem \ref{BUn}, that
$$H(BO(p_n))=\frac{H[u_1,{\dots},u_{2n},v_{2n+1}]}{(\tau v_{2n+1}+\{\alpha\} u_{2n},Ann(\{\alpha\})\cdot v_{2n+1})}$$
where $v_{2n+1}$ is nothing else but the element that maps to $\mu u_{2n}$ under the monomorphism $H(BO(p_n)) \rightarrow H({\mathfrak X}_{\alpha} \otimes BO(p_n))$. 
\end{proof}

At this point, we recall that ${\mathfrak X}_{\alpha} \otimes BO(p_n) \cong BO_{2n} \otimes {\mathfrak X}_{\alpha}$ by \cite[Proposition $2.6.1$]{SV}. Moreover, note that this isomorphism is functorial just by the way it is constructed in the proof of \cite[Proposition $2.6.1$]{SV}. In few words, for any torsor triple $(G,X,H)$ the isomorphism is obtained by considering the bisimplicial scheme $EG \times X \times EH$, then taking the quotient with respect to the left action of $G$ and the righ action of $H$ in different orders. The claimed isomorphism in $H_s(k)$ is 
$$\check{C}(X) \times BH \cong (G\backslash (EG \times X \times EH))/H=G\backslash ((EG \times X \times EH)/H) \cong BG \times \check{C}(X)$$
where $\check{C}(X)$ is the \v{C}ech simplicial scheme of $X$. So, a morphism of torsor triples $(G,X,H) \rightarrow (G',X',H')$ induces a commutative diagram in $H_s(k)$
$$
\xymatrix{
	\check{C}(X) \times BH \ar@{<->}[r] \ar@{->}[d] & BG \times \check{C}(X) \ar@{->}[d]\\
	\check{C}(X') \times BH' \ar@{<->}[r] & BG' \times \check{C}(X')
}
$$
where the horizontal maps are isomorphisms.

\begin{prop}\label{pq}
The isomorphism in motivic cohomology
$$H({\mathfrak X}_{\alpha} \otimes BO(p_n)) \longleftrightarrow H(BO_{2n} \otimes {\mathfrak X}_{\alpha})$$
induced by the isomorphism ${\mathfrak X}_{\alpha} \otimes BO(p_n) \cong BO_{2n} \otimes {\mathfrak X}_{\alpha}$ maps $u_{2i}$ to $u_{2i}$ and $u_{2i-1}$ to $u_{2i-1}+\mu u_{2i-2}$ for any $1 \leq i \leq n$.
\end{prop}
\begin{proof}
We proceed by induction on $n$. For $n=1$, by applying the argument just before this proposition to the morphism of torsor triples $(O(\la -1 \ra),Iso(\la -1 \ra \leftrightarrow \la -\alpha \ra),O(\la -\alpha \ra)) \rightarrow (O_2,Iso(q_2 \leftrightarrow \lva \alpha \rva),O(\lva \alpha \rva))$, we have the following commutative diagram
$$
\xymatrix{
H({\mathfrak X}_{\alpha} \otimes BO(\lva \alpha \rva)) \ar@{<->}[r] \ar@{->}[d] & H(BO_{2} \otimes {\mathfrak X}_{\alpha}) \ar@{->}[d]\\
H({\mathfrak X}_{\alpha} \otimes BO(\la -\alpha \ra)) \ar@{<->}[r] & H(BO({\la -1 \ra}) \otimes {\mathfrak X}_{\alpha})
}
$$
where the bottom horizontal isomorphism maps $u_1$ to $u_1+\mu$ by \cite[Proposition $2.6.1$ and Lemma $3.2.6$]{SV}. Then, the result follows from the fact that $u_1$ and $u_2$ are uniquely determined both in $H({\mathfrak X}_{\alpha} \otimes BO(\lva \alpha \rva))$ and in $H(BO_{2} \otimes {\mathfrak X}_{\alpha})$ by the fact that $u_1$ restricts to $u_1$ and $u_2$ vanishes respectively in $H({\mathfrak X}_{\alpha} \otimes BO(\la -\alpha \ra))$ and in $H(BO({\la -1 \ra}) \otimes {\mathfrak X}_{\alpha})$.

Now, suppose the statement is true for $n-1$. Then, the chain of morphisms of torsor triples $(O_{2n-2},Iso(q_{2n-2} \leftrightarrow p_{n-1}),O(p_{n-1})) \rightarrow (O(-q_{2n-1}),Iso(-q_{2n-1} \leftrightarrow p_{n-{\frac 1 2}}),O(p_{n-{\frac 1 2}})) \rightarrow (O_{2n},Iso(q_{2n} \leftrightarrow p_n),O(p_n))$ induces the following commutative diagram 
$$
\xymatrix{
H({\mathfrak X}_{\alpha} \otimes BO(p_n)) \ar@{<->}[r] \ar@{->}[d] & H(BO_{2n} \otimes {\mathfrak X}_{\alpha}) \ar@{->}[d]\\
H({\mathfrak X}_{\alpha} \otimes BO(p_{n-{\frac 1 2}})) \ar@{<->}[r] \ar@{->}[d] & H(BO(-q_{2n-1}) \otimes {\mathfrak X}_{\alpha}) \ar@{->}[d]\\
H({\mathfrak X}_{\alpha} \otimes BO(p_{n-1})) \ar@{<->}[r] & H(BO_{2n-2} \otimes {\mathfrak X}_{\alpha})
}
$$
In this case we need to understand first the homomorphism 
$$H(BO(p_{n-{\frac 1 2}})) \rightarrow H(BO(p_{n-1}))$$
In order to do so, we notice that
$$O(p_{n-{\frac 1 2}})/O(p_{n-1}) \cong A_{p_{n-{\frac 1 2}}=-1} \cong A_{\la \alpha \ra \perp (n-1)\HF=1}$$
From $\widetilde{M}(A_{\alpha x^2=1})=\n$ we deduce that 
$$\widetilde{M}(A_{\la \alpha \ra \perp (n-1)\HF=1})=\n(n-1)[2n-2]$$
Therefore, by Proposition \ref{Thom1} we have a long exact sequence in motivic cohomology
$${\dots} \rightarrow H^{*-1,*'}(BO(p_{n-1})) \xrightarrow{h^*} H^{*-2n+1,*'-n+1}(\n \otimes BO(p_{n-{\frac 1 2}})) \xrightarrow{f^*}$$
$$H^{*,*'}(BO(p_{n-{\frac 1 2}})) \xrightarrow{g^*} H^{*,*'}(BO(p_{n-1})) \rightarrow {\dots}$$
At this point, notice that $H(\n \otimes BO(p_{n-{\frac 1 2}}))=H(\n) \otimes_H H[u_1,\dots,u_{2n-1}]$ which implies that $u_i$ and $v_{2n-1}$ are all uniquely liftable to $H(BO(p_{n-{\frac 1 2}}))$ by degree reasons, since $H^{*,*'}(\n \otimes BO(p_{n-{\frac 1 2}}))$ is $0$ for $*'<0$ and for $(*')[*]=(0)[0]$ and $(0)[1]$, and by Lemma \ref{HMat}. Hence, $g^*$ is an epimorphism since it is a ring homomorphism. Moreover, $g^*(u_i)=u_i$ for $i \leq 2n-2$ since the natural restriction $H(BO(p_{n-{\frac 1 2}})) \rightarrow H(BO(p_{n-{\frac 3 2}}))$ factors through $H(BO(p_{n-1}))$ and the classes $u_i$ are uniquely determined, both in $H(BO(p_{n-{\frac 1 2}}))$ and in $H(BO(p_{n-1}))$, by the fact that they restrict to the respective $u_i$ or vanish for $i=2n-2$ in $H(BO(p_{n-{\frac 3 2}}))$. For the same reason, since $v_{2n-1}$ vanishes in $H(BO(p_{n-{\frac 3 2}}))$, the element that covers $v_{2n-1}$ through $g^*$ has the shape $u_{2n-1} + \epsilon u_1 u_{2n-2}$, where $\epsilon$ is $0$ or $1$. Suppose $\epsilon = 1$, then by \cite[Proposition 3.1.12]{SV} we have that $Sq^1(u_{2n-1} + u_1 u_{2n-2})=Sq^1Sq^1u_{2n-2}=0$, so $Sq^1v_{2n-1}=0$ as well. But, $Sq^1v_{2n-1}$ maps to $Sq^1(\mu u_{2n-2})$ in $H({\mathfrak X}_{\alpha} \otimes BO(p_{n-1}))$, which again maps to $Sq^1(\mu u_{2n-2})=\mu^2u_{2n-2}+\mu u_1u_{2n-2} \neq 0$ in $H(BO_{2n-2} \otimes {\mathfrak X}_{\alpha})=H({\mathfrak X}_{\alpha})[u_1,\dots,u_{2n-2}]$, and we get a contradiction. Hence, $\epsilon$ must be $0$ and $g^*(u_{2n-1})=v_{2n-1}$.

Therefore, we have that the isomorphism $H({\mathfrak X}_{\alpha} \otimes BO(p_{n-{\frac 1 2}})) \leftrightarrow H(BO(-q_{2n-1}) \otimes {\mathfrak X}_{\alpha})$ maps $u_{2i}$ to $u_{2i}$ and $u_{2i-1}$ to $u_{2i-1}+\mu u_{2i-2}$ for any $1 \leq i \leq 2n-2$. Moreover, since $H({\mathfrak X}_{\alpha} \otimes BO(p_{n-{\frac 1 2}})) \rightarrow H({\mathfrak X}_{\alpha} \otimes BO(p_{n-1}))$ maps $u_{2n-1}+\mu u_{2n-2}$ to $0$, we have that $H({\mathfrak X}_{\alpha} \otimes BO(p_{n-{\frac 1 2}})) \leftrightarrow H(BO(-q_{2n-1}) \otimes {\mathfrak X}_{\alpha})$ maps $u_{2n-1}$ to $u_{2n-1}+\mu u_{2n-2}$.

Now, the result follows from the fact that the $u_i$ are uniquely determined both in $H({\mathfrak X}_{\alpha} \otimes BO(p_n))$ and in $H(BO_{2n} \otimes {\mathfrak X}_{\alpha})$ by the fact that they restrict to $u_i$ for $i \leq 2n-1$ and vanishes for $i=2n$ respectively in $H({\mathfrak X}_{\alpha} \otimes BO(p_{n-{\frac 1 2}}))$ and in $H(BO(-q_{2n-1}) \otimes {\mathfrak X}_{\alpha})$. 
\end{proof}

From Theorem \ref{Hpn} we get immediately the following result which provides the motivic cohomology ring of $BO(\widetilde{h}_n)$.

\begin{thm}
For any $n \in \Z_{\geq 0}$ there exist cohomology classes $u_i$ of bidegree $([i/2])[i]$ for $1 \leq i \leq 2n$ and a class $v_{2n+1}$ of bidegree $(n)[2n+1]$ only for $n$ odd such that the motivic cohomology ring of $BO(\widetilde{h}_n)$ is given by 
$$H(BO(\widetilde{h}_n))=
\begin{cases}
\frac{H[u_1,{\dots},u_{2n},v_{2n+1}]}{(\tau v_{2n+1}+\{\alpha\} u_{2n},Ann(\{\alpha\})\cdot v_{2n+1})}, & n \: odd \\
H[u_1,{\dots},u_{2n}], & n \: even
\end{cases}
$$
\end{thm}
\begin{proof}
It follows from the fact that $\widetilde{h}_n$ is split for $n$ even and is isomorphic to $p_n$ for $n$ odd. 
\end{proof}

Once we know both the motivic cohomology of $BO(\widetilde{h}_n)$ and $BU_n(E/k)$, we can relate the subtle classes arising from the orthogonal group and those arising from the unitary group. In particular, we have the following result.

\begin{prop}\label{comp}
For any $n \in \Z_{\geq 0}$ the natural embedding $U_n(E/k) \hookrightarrow O(\widetilde{h}_n)$ induces an epimorphism 
$$H(BO(\widetilde{h}_n)) \rightarrow H(BU_n(E/k))$$ 
sending $u_{2i}$ to $c_i$ for any $1 \leq i \leq n$, $u_{2l+1}$ to $0$ for any $0 \leq l \: even < n$, $u_{2j+1}$ to $d_j$ for any $1 \leq j \: odd < n$ and $v_{2n+1}$ to $d_n$ only for $n$ odd.
\end{prop}
\begin{proof}
We will proceed by induction. Notice that the induction basis is provided by the isomorphism $U_0 \cong O_0 \cong Spec(k)$.

For $n$ odd we have the following commutative diagrams
$$
\xymatrixrowsep{1.5pc}\xymatrix{
U_{n-1}(E/k) \ar@{->}[r] \ar@{->}[d] & U_n(E/k) \ar@{->}[dd]\\
O_{2n-2} \ar@{->}[d] \\
O(\widetilde{h}_{n-{\frac 1 2}}) \ar@{->}[r]  & O(\widetilde{h}_n) 
}
\hspace{2cm}
\xymatrixrowsep{1.5pc}\xymatrix{
H(BU_{n-1}(E/k))  & H(BU_n(E/k)) \ar@{->}[l] \\
H(BO_{2n-2}) \ar@{->}[u] \\
H(BO(\widetilde{h}_{n-{\frac 1 2}})) \ar@{->}[u]  & H(BO(\widetilde{h}_n)) \ar@{->}[l] \ar@{->}[uu]
}
$$
where by $\widetilde{h}_{n-{\frac 1 2}}$ here we mean the quadratic form $\la -\alpha \ra \perp (n-1)\HF$.

By induction hypothesis we have that $u_{2i}$ goes to $c_i$ for any $1 \leq i \leq n-1$, $u_{2l+1}$ to $0$ for any $0 \leq l \: even < n-1$ and $u_{2j+1}$ to $d_j$ for any $1 \leq j \: odd < n-1$. The class $u_{2n-1}$ goes to $0$ via the map $H(BO(\widetilde{h}_n)) \rightarrow H(BU_{n-1}(E/k))$ since this factors through $H(BO_{2n-2})$. Hence, $u_{2n-1}$ maps to $0$ in $H(BU_n(E/k))$ since the morphism $H(BU_n(E/k)) \rightarrow H(BU_{n-1}(E/k))$ is injective in bidegree $(n-1)[2n-1]$. Moreover, noticing that 
$$U_n(E/k)/U_{n-1}(E/k) \cong O(\widetilde{h}_n)/O(\widetilde{h}_{n-{\frac 1 2}})$$ 
and by Proposition \ref{Thom2}, we obtain that $u_{2n}$ goes to $c_n$ and $v_{2n+1}$ goes to $d_n$.

For $n$ even we have similar commutative diagrams
$$
\xymatrixrowsep{1.5pc}\xymatrix{
U_{n-1}(E/k) \ar@{->}[r] \ar@{->}[d] & U_n(E/k) \ar@{->}[dd]\\
O(\widetilde{h}_{n-1}) \ar@{->}[d] \\
O_{2n-1} \ar@{->}[r]  & O_{2n}
}
\hspace{2cm}
\xymatrixrowsep{1.5pc}\xymatrix{
H(BU_{n-1}(E/k))  & H(BU_n(E/k)) \ar@{->}[l] \\
H(BO(\widetilde{h}_{n-1})) \ar@{->}[u] \\
H(BO_{2n-1}) \ar@{->}[u]  & H(BO_{2n}) \ar@{->}[l] \ar@{->}[uu]
}
$$
In this case we need to study the homomorphism 
$$H(BO_{2n-1}) \rightarrow H(BO(\widetilde{h}_{n-1}))$$
In order to do so, we notice that
$$O_{2n-1}/O(\widetilde{h}_{n-1}) \cong A_{\widetilde{h}_{n-1} \perp \la \alpha \ra=\alpha}\cong A_{\alpha^{-1} \widetilde{h}_{n-1} \perp \la 1 \ra=1}$$
From $\widetilde{M}(A_{x^2=\alpha})=\n$ and since $\alpha^{-1} \widetilde{h}_{n-1} \perp \la 1 \ra$ is isomorphic to $\la \alpha^{-1} \ra \perp (n-1)\HF$, we deduce that 
$$\widetilde{M}(A_{\alpha^{-1} \widetilde{h}_{n-1} \perp \la 1 \ra=1})=\n(n-1)[2n-2]$$
Hence, by Proposition \ref{Thom1} we have a long exact sequence in motivic cohomology
$${\dots} \rightarrow H^{*-1,*'}(BO(\widetilde{h}_{n-1})) \xrightarrow{h^*} H^{*-2n+1,*'-n+1}(\n \otimes BO_{2n-1}) \xrightarrow{f^*}$$ 
$$H^{*,*'}(BO_{2n-1}) \xrightarrow{g^*} H^{*,*'}(BO(\widetilde{h}_{n-1})) \rightarrow {\dots}$$
Then, by repeating exactly the same arguments that appear in Proposition \ref{pq} we get that $g^*(u_i)=u_i$ for $i \leq 2n-2$ and $g^*(u_{2n-1})=v_{2n-1}$.
 
Therefore, by induction hypothesis we have that $u_{2i}$ goes to $c_i$ for any $1 \leq i \leq n-1$, $u_{2l+1}$ to $0$ for any $0 \leq l \: even \leq n-1$ and $u_{2j+1}$ to $d_j$ for any $1 \leq j \: odd \leq n-1$. Moreover, recalling that 
$$U_n(E/k)/U_{n-1}(E/k) \cong O_{2n}/O_{2n-1}$$ 
and by Proposition \ref{Thom2}, we obtain that $u_{2n}$ goes to $c_n$, as we aimed to show. 
\end{proof}

As a corollary of the previous proposition and of Theorem \ref{BUn} we get a description of $H(BU_n(E/k))$ as a quotient of $H(BO(\widetilde{h}_n))$.

\begin{cor}\label{cor}
For any $n \in \Z_{\geq 0}$ there is an isomorphism
$$H(BU_n(E/k)) \cong {\frac {H(BO(\widetilde{h}_n))} {R}}$$
where $R$ is the ideal generated by $u_{4j+1}$, $u_{4i+3}u_{4j+2}+u_{4j+3}u_{4i+2}$, $\tau u_{4j+3}+\{\alpha\} u_{4j+2}$ and $Ann(\{\alpha\}) \cdot u_{4j+3}$ for any $0 \leq i,j \leq [{\frac {n-1} {2}}]$, where $u_{2n+1}$ is substituted by $v_{2n+1}$ for $n$ odd.
\end{cor} 

\section{Applications to Hermitian forms}

Throughout this section we exploit previous results to study subtle Stiefel-Whitney classes of quadratic forms divisible by $\lva \alpha \rva$. The general idea is that $H(BU_n(E/k))$ is closer to the cohomology of the \v{C}ech simplicial scheme of a quadratic form associated to a hermitian form than $H(BO_{2n})$.

From $\cite{SV}$, we know that for every hermitian form $h$ of the quadratic extension $E/k$ there exists a commutative diagram
$$
\xymatrix{
\check C(X_h) \ar@{->}[r] \ar@{->}[d] & BU_n(E/k) \ar@{->}[r] \ar@{->}[d] & BO(\widetilde{h}_n) \ar@{->}[d]\\
Spec(k) \ar@{->}[r]^(0.42){h} & B_{et}U_n(E/k) \ar@{->}[r] & B_{et}O(\widetilde{h}_n)
}
$$
where $\check C(X_h)$ is the \v{C}ech simplicial scheme of the torsor $X_h=Iso(h \leftrightarrow h_n)$. Hence, the computation of the motivic cohomology of $BU_n(E/k)$ provides us with subtle characteristic classes for hermitian forms and relations among them. More precisely, we have the following proposition.

\begin{prop}
For any $n$-dimensional hermitian form $h$, in $H(\check C(X_h))$ the following relations hold for any $1 \leq j,j' \: odd \leq n$:\\
1) $c_{j'}(h)d_j(h)+c_j(h)d_{j'}(h)=0$;\\
2) $\tau d_j(h)+\{\alpha\} c_j(h)=0$;\\
3) $Ann(\{\alpha\})\cdot d_j(h)=0$.
\end{prop}
\begin{proof}
It follows immediately from Theorem \ref{BUn}. 
\end{proof}

We now move to consider quadratic forms associated to hermitian ones and their subtle Stiefel-Whitney classes.

Recall that two hermitian forms are isomorphic if and only if the corresponding quadratic forms over $k$ are isomorphic. In particular, for even dimensional hermitian forms we have that they split if and only if the respective quadratic forms split. It follows that $\check C(X_h) \cong \check C(X_{\widetilde{h}})$, for even dimensional hermitian forms.

\begin{prop}
For $n$ even, in $H(\check C(X_{\widetilde{h}}))$ the following relations hold for any $0 \leq i,j \leq \frac{n}{2}-1$:\\
1) $u_{4j+1}({\widetilde{h}})=0$;\\
2) $u_{4i+3}({\widetilde{h}})u_{4j+2}({\widetilde{h}})=u_{4j+3}({\widetilde{h}})u_{4i+2}({\widetilde{h}})$;\\
3) $\tau u_{4j+3}({\widetilde{h}})=\{\alpha\}u_{4j+2}({\widetilde{h}})$;\\
4) $Ann(\{\alpha\})\cdot u_{4j+3}({\widetilde{h}})=0$.
\end{prop}
\begin{proof}
It follows immediately from Corollary \ref{cor}. 
\end{proof}

On the other hand, if $q$ is an odd dimensional quadratic form, then $\lva \alpha \rva \otimes q$ is split over a field extension of $k$ if and only if $\lva \alpha \rva$ is split over the same field extension. It follows from this remark that, for odd dimensional hermitian forms, $\check C(X_{\widetilde{h}}) \cong \check C(X_{\alpha})$, where $\check C(X_{\alpha})$ stands for the \v{C}ech simplicial scheme associated to the Pfister form $\lva \alpha \rva$.  

\begin{prop}
For $n$ odd, in $H(\check C(X_{\widetilde{h}}))=H(\check C(X_{\alpha}))$ the following relations hold for any $0 \leq j \leq \frac{n-1}{2}$:\\
1) $u_{4j+1}({\widetilde{h}})=\mu u_{4j}({\widetilde{h}})$;\\
2) $u_{4j-1}({\widetilde{h}})=0$.
\end{prop}
\begin{proof}
Together with the commutative diagram at the beginning of this section, we have the following one
$$
\xymatrix{
\check C(X_{\alpha}) \ar@{->}[r] \ar@{->}[d] & BO_{2n} \ar@{->}[d]\\
Spec(k) \ar@{->}[r]^{\widetilde{h}} & B_{et}O_{2n}
}
$$
By \cite[Proposition $2.6.1$]{SV}, we know that after tensoring both with $\check C(X_{\alpha})$ they coincide. Therefore, our restriction morphism $H(BO_{2n}) \rightarrow H(\check C(X_{\alpha}))$ factors as 
$$H(BO_{2n}) \rightarrow H(BO_{2n} \times \check C(X_{\alpha})) \leftrightarrow H(\check C(X_{\alpha}) \times BO(\widetilde{h}_n)) \rightarrow H(\check C(X_{\alpha}) \times BU_n(E/k)) \rightarrow H(\check C(X_{\alpha}))$$
which implies the result by Theorem \ref{main}, Proposition \ref{pq} and Proposition \ref{comp}. 
\end{proof}

We now show that the subtle classes arising in the unitary case see the triviality of the torsor of a hermitian form in the same way subtle Stiefel-Whitney classes do for quadratic forms (\cite[Corollary 3.2.32]{SV}).

\begin{prop}
$h \cong h_n$ if and only if $c_{2^r}(h)=0$ for any $r$.
\end{prop}
\begin{proof}
Let us start from the case $n$ even. Then, we have already noticed that $h$ splits if and only if $\widetilde{h}$ splits. This is equivalent to say that $u_{2^{r+1}}(\widetilde{h})$ vanishes in $H(\check C(X_{\widetilde{h}}))$ for any $r$, which is the same of vanishing of $c_{2^r}(h)$ in $H(\check C(X_h))$, since in this case $\check C(X_h) \cong \check C(X_{\widetilde{h}})$ and by Proposition \ref{comp}.

For $n$ odd, we have that $h$ splits if and only if $h \perp \la -1 \ra$ (which is even dimensional) splits. This amounts to say that $c_{2^r}(h \perp \la -1 \ra)=0$ in $H(\check C(X_{h \perp \la -1 \ra}))$ for any $r$, which is equivalent to say that $c_{2^r}(h)=0$ in $H(\check C(X_h))$ for any $r$. 
\end{proof}

We conclude by presenting an expression of the motive of the torsor associated to a hermitian form. Indeed, by the very same arguments of \cite[Propositions 3.1.11 and 3.2.2]{SV} one obtains the description of the motive of the torsor $X_h$ in terms of motives of \v{C}ech simplicial schemes and subtle characteristic classes, where $h$ is any hermitian form.

Before stating the results, let us denote by $\widetilde{c_j}$ a morphism $T \rightarrow \N(j)[2j]$ in ${\mathcal DM}^{-}_{eff}(BU_n(E/k))$ which composed with the only non-zero morphism $\N(j)[2j] \rightarrow T(j)[2j]$ gives $c_j$ for any $j$ odd. It is actually the unique cohomology class in $H(\n \otimes BU_n(E/k))$ that maps to $c_j$ under the homomorphism induced by the only non-zero morphism $T \rightarrow \n$. Then, we have the following two propositions.

\begin{prop}
In ${\mathcal DM}^{-}_{eff}(BU_n(E/k))$ we have
$$M(EU_n(E/k) \rightarrow BU_n(E/k))=\bigotimes\limits_{1 \leq i \: even \leq n}Cone[-1](T \xrightarrow{c_i} T(i)[2i]) \otimes \bigotimes\limits_{1 \leq j \: odd \leq n}Cone[-1](T \xrightarrow{\widetilde{c_j}} \N(j)[2j])$$
\end{prop}

\begin{prop}
In ${\mathcal DM}^{-}_{eff}(k)$ we have
$$M(X_h)=\bigotimes\limits_{1 \leq i \: even \leq n}Cone[-1]({\mathfrak X}_h \xrightarrow{c_i(h)} {\mathfrak X}_h(i)[2i]) \otimes \bigotimes\limits_{1 \leq j \: odd \leq n}Cone[-1]({\mathfrak X}_h \xrightarrow{\widetilde{c_j}(h)} \N \otimes {\mathfrak X}_h(j)[2j])$$
\end{prop}

\footnotesize{
}

\noindent {\scshape School of Mathematical Sciences, University of Nottingham}\\
fabio.tanania@gmail.com

\end{document}